\documentclass[12pt, letterpaper]{article}
\usepackage{amsfonts}
\usepackage{amssymb}
\usepackage{latexsym}  
\usepackage{pifont}  
\usepackage{bbm}  %

\usepackage{stmaryrd}  

\let\SavedRightarrow=\Rightarrow
\usepackage{marvosym}
\let\Rightarrow=\SavedRightarrow

\usepackage{amsmath}   

\newenvironment{itemizz}{\begin{itemize}\setlength{\itemsep}{-1mm}}
{\end{itemize}}
\newenvironment{enum}{\begin{enumerate}\setlength{\itemsep}{-1mm}}
{\end{enumerate}}

{\begin{itemize} \setlength{\itemsep}{-1mm} %
\renewcommand\labelitemi{\ding{#1}}} %
{\end{itemize}}

\newtheorem{theorem}{Theorem}[section]
\newtheorem{definition}[theorem]{Definition}
\newtheorem{lemma}[theorem]{Lemma}

\newtheorem{example}[theorem]{Example}
\newtheorem{proposition}[theorem]{Proposition}
\newtheorem{question}[theorem]{Question}

\newcommand\RRR{\mathbb {R}}
\newcommand\CCC{\mathbb {C}}
\newcommand\PPP{\mathbb {P}}

\newcommand\ZZZ{\mathbb {Z}}
\newcommand\QQQ{\mathbb {Q}}

\newcommand\PP{\mathcal {P}}

\newcommand\FF{\mathcal {F}}
\newcommand\HH{\mathcal {H}}

\newcommand\NN{\mathcal {N}}

\newcommand\WW{\mathcal {W}}
\newcommand\MM{\mathcal {M}}

\newcommand\Ups{\Upsilon}

\newcommand\twist{\mathrm{twist}}  
  
\newcommand\tw{\mathrm{tw}}  
\newcommand\dom{\mathrm{dom}}  
\newcommand\cl{\mathrm{cl}}  
\newcommand\ran{\mathrm{ran}}   
\newcommand\MA{\mathrm{MA}}  
\newcommand\PFA{\mathrm{PFA}}  
\newcommand\ZFC{\mathrm{ZFC}}  
\newcommand\CH{\mathrm{CH}}  
\newcommand\GCH{\mathrm{GCH}}  
\newcommand\hgt{\mathrm{ht}}   
\newcommand\SO{\mathrm{SO}}   

\newcommand\cccc{\mathfrak{c}}

\newcommand{\pref}[1]{(\ref{#1})}

\newcommand\res{\mathord {\upharpoonright}}  

\newcommand\iv{^{-1}} 

\newcommand\one{\mathbbm{1}} 

\newcommand\compat{\mathrel{\not\perp}}  

\newcommand\eop{{\Large \Coffeecup}}  

\newenvironment{proof}{{\bf Proof.}}{\eop\medskip}
\newenvironment{proofof}[1]{\medskip \textbf{Proof of #1.}}{\eop\medskip}

\begin{document}

\title{Homeomorphisms with Small Twist%
}

\author{Kenneth Kunen\footnote{University of Wisconsin,  Madison, WI  53706, U.S.A.,
\ \ kunen@math.wisc.edu} }

\maketitle

\begin{abstract}
We extend Baumgartner's result on isomorphisms of $\aleph_1$--dense
subsets of $\RRR$ in two ways:  First, the function can be made
to be absolutely continuous.  Second, one can replace $\RRR$ by $\RRR^n$.
\end{abstract}

\section{Introduction}
\begin{definition}
For any topological space $X$, $\HH(X)$ denotes the set of
all homeomorphisms from $X$ onto $X$,
and a subset
$A \subseteq X$ is \emph{$\kappa$--dense} \textup(in $X$\textup)
iff $|A \cap U| = \kappa$
for all non-empty open $U \subseteq X$.
\end{definition}

Then, for $X = \RRR$, we have

\begin{theorem}
\label{thm-baum-cant}
{$  $}
\begin{itemizz}
\item[a.] If $D,E $ are $\aleph_0$--dense in $\RRR$,
then there is an $f \in \HH(\RRR)$ such that $f(D)  = E$.
\item[b.] Assuming $\PFA$, if $D,E $ are $\aleph_1$--dense in $\RRR$,
then there is an $f \in \HH(\RRR)$ such that $f(D)  = E$.
\end{itemizz}
\end{theorem}

Here, (a) is a classical result of Cantor, while (b)
is due to Baumgartner \cite{Ba1,Ba2}.
In both cases, the proof obtains an order isomorphism $h$ from
$D$ onto $E$, which must then extend to a unique $f \in \HH(\RRR)$.
In (b), Baumgartner's original proof \cite{Ba1} predates $\PFA$;
he simply showed that the result of the theorem,
together with $\MA + \cccc = \aleph_2$, can be
obtained by iterated ccc forcing over any model of $\ZFC + \GCH$.
Using his forcing, the $\PFA$ result is immediate by
the ``collapse the continuum trick'' (see \cite{Ba2}) ; similar remarks hold
for our uses of $\PFA$ in this paper.

By Avraham and  Shelah \cite{AS},
the result in (b) does not follow from $\MA + \cccc=\aleph_2$ alone.

In this paper, we assume $\PFA$ and prove two extensions of (b).
First, we show that both $f$ and $f\iv$ can be made to be 
absolutely continuous (AC).  Absolute continuity for real-valued
functions is discussed below, and in many analysis texts,
such as Rudin \cite{Rud}.
It is easily seen (Example \ref{ex-not-ac} below)
that Baumgartner's forcing
yields an $f$ such that neither $f$ nor $f\iv$ is AC.
If $f$ is Lipschitz 
(\,$\forall x,z \, [\, |f(x) - f(z)| \le C |x - z |\, ]$\,),
then $f$ must be AC, but one cannot improve (b) to make
$f$ and $f\iv$ Lipschitz; a $\ZFC$
counter-example is described in \cite{Kunen2}, although this example is
implicit in the earlier \cite{ARS}.
Note that in (a), it is easy to make $f$ and $f\iv$ Lipschitz, and
also real-analytic; this seems to have been done first 
by Franklin \cite{Franklin} in 1925.

Our second extension of (b) replaces $\RRR$ by $\RRR^n$.
One such extension is already known, and is due to
Stepr\=ans and Watson \cite{SW}:

\begin{theorem}
\label{thm-st-wat}
For any infinite $\kappa$ and any finite $n \ge 2$,
$\MA(\kappa)$ implies that
if $D,E $ are $\kappa$--dense in $\RRR^n$,
then there is an $f \in \HH(\RRR^n)$ such that $f(D)  = E$.
\end{theorem}

This makes it appear that the result for  $\RRR^n$, for $n \ge 2$,
is ``easier'' than for $\RRR$.  When $\kappa = \aleph_1$,
we only need $\MA + \cccc=\aleph_2$, not $\PFA$.
When $\kappa = \aleph_2$ and $n=1$,
it is a well-known open question whether the result of Theorem \ref{thm-st-wat}
is even \emph{consistent} with $\cccc\ge\aleph_2$.

The ``easiness'' of $\RRR^n$ for $n \ge 2$ is explained by the
fact that $\RRR^n$ has ``more'' homeomorphisms than $\RRR$.
For example, every permutation of a finite
subset of $\RRR^n$ extends to some $f \in \HH(\RRR^n)$, while
this is clearly false for $n = 1$, since
every $f \in \HH(\RRR)$ is monotonic (either order-preserving
or order-reversing); in fact, the proofs of (a) and (b) 
in Theorem \ref{thm-baum-cant} produce order-preserving functions.
Now,  if we set $\kappa = \aleph_1$ and
demand that our $f$ in Theorem \ref{thm-st-wat} be 
``order-preserving'' (\emph{suitably defined}), then we do 
get a harder result that follows from $\PFA$ but not from
$\MA(\aleph_1)$.  As with the $n = 1$ results, we do not
know if there is any consistent version of our results
with $\kappa > \aleph_1$.

 But, what is the right definition of
``order-preserving''?  One possibility might be order-preserving on
each coordinate; i.e., for each $\vec x, \vec z \in \RRR^n$,
and each coordinate $i = 0,\ldots, n-1$:
$x_i < y_i$ iff $ f(x_i) < f(y_i)$ for all $i$.
But this is ``wrong'', in that there is a $\ZFC$ counter-example
in $\RRR^2$ (Example  \ref{ex-no-op}).
A ``correct'' definition, which leads to a $\PFA$ theorem,
involves the notion of \emph{twist}:

\begin{definition}
\label{def-angle}
For $\vec v, \vec w  \in \RRR^n \backslash \{ \vec 0 \}$:
\[
\angle(\vec v, \vec w) =
\arccos(\, (\vec v \cdot \vec w) / (\|\vec v\| \|\vec w\|)\,)
\in [0,\pi] \ \ .
\]
\end{definition}

So, we are thinking of $\vec v, \vec w$
as arrows pointing from the origin $\vec 0$,
and we are measuring the angle between them in the usual way.

\begin{definition}
\label{def-twist}
If $F \subseteq \RRR^n \times \RRR^n$, let 
\[
\twist(F) = \{ \angle( d_1 - d_0, e_1 - e_0 ) :
(d_0,e_0), (d_1,e_1) \in F
 \wedge
d_0 \ne d_1
 \wedge
e_0 \ne e_1
\} \ \ .
\]
Then, let $\tw(F) = \sup(\twist(F))$.
\end{definition}

In our applications, $F$ will usually be the graph of a bijection,
although $\dom(F)$ and $\ran(F)$ may be proper subsets of $\RRR^n$.

\begin{lemma}
\label{lemma-cl-tw}
For any $F \subseteq \RRR^n \times \RRR^n$:
$\twist(F) \subseteq [0,\pi]$, and
$\tw(F) \in [0,\pi]$, and
$\twist(\overline F) \subseteq \cl( \twist(F)) $, and
$\tw(\overline F) =  \tw(F) $.
\end{lemma}

When $n = 1$, $\twist(F) \subseteq \{0,\pi\}$, and a bijection
$F$ is strictly increasing (i.e., order-preserving) iff $\tw(F) = 0$.

Then we shall prove

\begin{proposition}
\label{prop-intro}
Assume $\PFA$.  Fix $\theta > \pi/2$ and $\aleph_1$--dense
$D,E \subset \RRR^n$.  Then there
is an $f \in \HH(\RRR^n)$
such that $f(D) = E$ and
$\tw(f) \le \theta$.
\end{proposition}

The ``$\PFA$'' is needed here, since it is consistent with
$\MA + \cccc = \aleph_2$ that the proposition fails for all 
$n \ge 1$ and all $\theta < \pi$ (Example \ref{ex-entangled}).

The ``$\theta > \pi/2$'' is needed here, since for 
$\theta \le \pi/2$ and $n \ge 2$, there is a $\ZFC$ counter-example
(Example \ref{ex-twistlimit}).
Of course, when $n = 1$, this is just Baumgartner's result,
and $\tw(f)$ can be $0$.

But now, we wish to add into
Proposition \ref{prop-intro} the claim that $f$ is AC.
Since for $n \ge 2$, AC is not quite a standard notion,
we shall define what we mean here:

\begin{definition}
\label{def-BAC}
Let $X$ be a Polish space with a $\sigma$-finite Borel measure $\mu$,
and fix $f \in \HH(X)$.
Then $f$ is \emph{absolutely continuous}
\textup(with respect to $\mu$\textup)
iff for all
$\varepsilon > 0$ there is a $\delta > 0$ such that 
for all open $U$, 
$\mu(U)  < \delta \to \mu(f(U)) < \varepsilon$.
$f$ is \emph{bi-absolutely continuous} \textup(BAC\textup) iff
$f$ and $f\iv$ are both absolutely continuous.
When discussing $\RRR^n$, $\mu$ always refers to Lebesgue measure.
\end{definition}

When $X = \RRR$, $f$ is a monotonic function, and this definition
coincides with the usual definition of absolute continuity
for real-valued functions.
For general $X$ and $f$: If $f$ is BAC, then the induced measures are
absolutely continuous 
($\mu \ll \mu f \ll \mu f\iv \ll \mu$;
that is, $\mu(B) = 0 \leftrightarrow
\mu(f(B)) = 0 \leftrightarrow \mu(f\iv(B)) = 0 $
for all Borel $B \subseteq X$). 
This implication is an equivalence when $\mu(X) < \infty$, but
not in general; the map $x \mapsto x^3$ on $\RRR$ is a
counter-example.

We can now combine our two extensions of Baumgartner's result:

\begin{theorem}
\label{thm-main}
Assume $\PFA$.  Fix $\theta > \pi/2$ and $\aleph_1$--dense
$D,E \subset \RRR^n$.  Then there
is an $f \in \HH(\RRR^n)$
such that $f(D) = E$ and $\tw(f) \le \theta$ and $f$ is BAC.
\end{theorem}

Proposition \ref{prop-intro} is obvious from this.
Theorem \ref{thm-main} is proved at the end of Section \ref{sec-ac}.
We shall prove the $n = 1$ case first (Lemma \ref{lemma-main-1});
here, the ``$\tw(f) \le \theta$'' is trivial, making the proof
quite a bit simpler; we shall then use the notation in that proof
to motivate the terminology in the general proof.
Actually, our proof for the $n > 1$ case uses some properties
of our forcing poset that are not proved until Sections
\ref{sec-mat} and \ref{sec-ccc}.

\section{The Basic Poset}
\label{sec-basic}
We describe here a natural modification of 
Baumgartner's poset, obtained by replacing $\RRR$ by $\RRR^n$
and replacing ``order preserving'' by
a restriction on twists, and we shall prove
that our poset is ccc.
Since we plan to use $\PFA$ with the
``collapse the continuum trick''
(or else just do an iterated forcing argument over a model of $\GCH$),
it is sufficient to assume $\CH$, fix $\theta,D,E$,
and produce a ccc poset $\PPP$ that forces an appropriate $f$.  
For constructing ccc posets in our forcing arguments,
we use the standard setup with elementary submodels, following
approximately the terminology in \cite{Kunen1}:

\begin{definition}
\label{def-elem-submod-chain}
Let $D,E \subseteq \RRR^n$ be $\aleph_1$--dense.
Fix $\kappa$,  a suitably large regular cardinal.
Let $\langle M_\xi : 0 < \xi < \omega_1 \rangle$ be a continuous 
chain of countable elementary submodels of $H(\kappa)$,
with $D,E \in M_1$ and each $M_\xi \in M_{\xi + 1}$.
Let $M_0 = \emptyset$.
For $x \in \bigcup_\xi M_\xi$, let $\hgt(x)$, the \emph{height} of $x$,
be the $\xi$ such that $x \in M_{\xi+1} \backslash M_\xi$.
\end{definition}

By setting $M_0 = \emptyset$, we ensure that under $\CH$,
$\hgt(x)$ is defined whenever $x \in \RRR^n$ or 
$x$ is a Borel subset of $\RRR^n$.
Observe that $\{d \in D : \hgt(d) = \xi\}$ and
$\{e \in E : \hgt(e) = \xi\}$ are both countable and dense
for each $\xi < \omega_1$.
Note that $\hgt(\, (x,y) \,) = \max( \hgt(x), \hgt(y) )$.

\begin{definition}
\label{def-poset-approx}
Fix $\theta\in (0,\pi)$ and $\aleph_1$--dense
$D,E \subset \RRR^n$. 
Assume $\CH$ and
use the notation from Definition \ref{def-elem-submod-chain} for the elementary 
submodels.  Then, let $\PPP^\theta_0$ be the set of all 
$p$ satisfying:
\begin{itemizz}  
\item[P1.] $p \in
[ D  \times E ]^{<\omega}$ is a bijection from $\dom(p)$ onto
$\ran(p)$.
\item[P2.] $\tw(p) < \theta$.
\item[P3.] For each $(d,e) \in p$,
$\hgt(d), \hgt(e)$ differ by a finite non-zero ordinal.
\item[P4.]  $(d_0,e_0) \in p \wedge (d_1,e_1) \in p
\wedge (d_0,e_0)\ne(d_1,e_1) \Rightarrow
\hgt(\, (d_0,e_0)\,) \ne \hgt (\, (d_1,e_1)\,) $.
\end{itemizz}
Define $q \le p$ iff $q \supseteq p$; so $\one = \emptyset$.
When $n=1$,  $\PPP_0 =  \PPP^\theta_0$ for
some \textup(any\textup) $\theta \in (0,\pi)$.
\end{definition}

Consider the one-dimensional version of this, so
in the ground model $V$, $D, E$ are $\aleph_1$--dense subsets of $\RRR$.
It is easy to see that the sets
$\{p : d \in \dom(p)\}$ and $\{p : e \in \ran(p)\}$ are dense for
all $d \in D$ and $e \in E$, so in $V[G]$, $\bigcup G$
is an order-preserving bijection from $D$ onto $E$.
Viewing $\bigcup G$ as a subset of $\RRR \times \RRR$,
let $f = \cl(\bigcup G)$.  Then, in $V[G]$ we have
$f \in \HH(\RRR)$ and $f(D) = E$.

Since the definition of $\PPP_0$ contains nothing relevant to absolute
continuity, this cannot suffice
to prove Theorem \ref{thm-main}:

\begin{example}
\label{ex-not-ac}
With $f$ as above, neither $f$ nor $f\iv$
is absolutely continuous. 
\end{example}
\begin{proof}
First, for 
$p \in \PPP_0$, let $h_p \in \HH(\RRR)$ be the natural
piecewise linear extension
of $p$ obtained by linear interpolation, giving it a slope of $1$
outside of 
$[\min(\dom(p)), \max(\dom(p))]$.
Let $h_\one(x) = x$.
Note that $(h_p)\iv = h_{p\iv}$.
When $p \ne \one$, let
$d^0_p = \min(\dom(p))$ and $d^1_p = \max(\dom(p))$ and
$e^0_p = p(d^0_p) = \min(\ran(p))$ and
$e^1_p = p(d^1_p) = \max(\ran(p))$.
For each $n > 0$,
let $\Delta_n$ be the set of all $p$ such that
$d^0_p, e^0_p \le -n$ and $d^1_p, e^1_p\ge n$ and
$\forall x \in [(d^0_p, d^1_p) \backslash \dom(p) ] \;
[h_p'(x) \in (0, 2^{-n}) \cup (2^n, \infty) ]  $.
Note that all the $\Delta_n$ are dense.  Using these, and setting
$f = \cl(\bigcup G)$, we see that both $f$ and $f\iv$
map a null set onto the complement of a null set.
\end{proof}

Also, both $f'$ and $(f\iv)'$ are differentiable almost everywhere,
with derivative $0$ almost everywhere.

We shall eventually modify $\PPP^\theta_0$ by adding some side conditions, 
obtaining a proof of Theorem \ref{thm-main},
but we shall conclude this section by proving
that $\PPP^\theta_0$ is ccc.
This is a straightforward variant of Baumgartner's argument:

\begin{lemma}
\label{lemma-elem-ccc}
Fix $\theta > \pi/2$ and $t \in \omega$, and assume that:
\begin{itemizz}
\item[1.]
$p_\alpha = \{ (d_\alpha^0, e_\alpha^0 ), \ldots,
(d_\alpha^{t-1} , e_\alpha^{t-1} ) \} $ satisfies $(P1)(P3)(P4)$ above
for each $\alpha < \omega_1$.
\item[2.]
$ d_\alpha^i \ne \ d_\beta^j$ and
$ e_\alpha^i \ne \ e_\beta^j$ unless
$\alpha = \beta$ and $i = j$.
\end{itemizz}
Then there are $\alpha \ne \beta$ such that 
$\angle( d_\beta^i - d_\alpha^i, e_\beta^i - e_\alpha^i ) < \theta$
for all $i < t$.
Hence, $\PPP^\theta_0$ is ccc.
\end{lemma}

\begin{proof}
The ccc follows from the rest of the lemma by a standard delta system argument.

Now, induct on $t$.  The case $t =0$ is trivial, so assume the
result for $t$, and we shall prove it for $t+1$; so now $p_\alpha =
\{ (d_\alpha^0, e_\alpha^0 ), \ldots, (d_\alpha^t , e_\alpha^t ) \} $.
Permuting and thinning the sequence if necessary, we may assume that each
$\hgt(p_\alpha) = \hgt(e_\alpha^t) > \hgt(d_\alpha^t)$,
and that $\alpha < \beta \to \hgt(p_\alpha) < \hgt(p_\beta)$.
Note that
$\hgt(p_\alpha) > \hgt(d_\alpha^{i})$ and
$\hgt(p_\alpha) > \hgt(e_\alpha^{i})$ for all $i < t$.

Identify each $p_\alpha$ with a point in $(\RRR^n)^{2t+2}$,
and let $K = \cl\{ p_\alpha : \alpha < \omega_1\} \subseteq (\RRR^n)^{2t+2}$.
For each $\alpha$  and each $y \in \RRR^n$, obtain
$p_\alpha/y \in (\RRR^n)^{2t+2}$
by replacing the $e_\alpha^{t}$ by $y$ in $p_\alpha$.
Let $K_\alpha = \{y \in \RRR^n : p_\alpha/y \in K\}$.
Applying $\CH$, fix $\zeta$ such that $K \in M_\zeta$.

For $\alpha \ge \zeta$: $K_\alpha$ is uncountable because
$K_\alpha \in M_{\hgt(p_\alpha)}$,
$e_\alpha^{t} \in K_\alpha$, and $e_\alpha^t \notin M_{\hgt(p_\alpha)}$.
Fix $\widehat e_\alpha \ne \widetilde e_\alpha$  in $ K_\alpha$;
we may assume that these are different from all the $ e_\alpha^i $.
Since $\theta > \pi/2$,
$\varepsilon := \theta - \pi/2 > 0$.
Now, fix disjoint basic open neighborhoods $U, V$ of
$\widehat e_\alpha, \widetilde e_\alpha$ respectively
so that
$\angle ( x_1 -  y_1,\, x_2 - y_2) < \varepsilon/2$
for all $ x_1, x_2 \in U$ and all
$ y_1, y_2 \in V$.

Of course, $U, V$ depend on $\alpha$, but we may fix an
uncountable $S \subseteq \omega_1 \backslash \zeta$
such that they have the same values for all $\alpha \in S$.
Then, applying induction, fix $\alpha \ne \beta$ in $S$ such that
$\angle( d_\beta^i - d_\alpha^i, e_\beta^i - e_\alpha^i ) < \theta$
for all $i < t$.  Then, fix any $x \in U$ and any $y \in V$.
Then either $\angle( d_\beta^t - d_\alpha^t, y-x) \le \pi/2$ or
$\angle( d_\beta^t - d_\alpha^t, x-y) \le \pi/2$, since the
sum of the two angles is $\pi$.
In any case, $x, \widehat e_\alpha, \widehat e_\beta \in U$ and
$y, \widetilde e_\alpha, \widetilde e_\beta \in V$.

If  $\angle( d_\beta^t - d_\alpha^t, y-x) \le \pi/2$, use
$\widehat e_\alpha \in K_\alpha$ and $\widetilde e_\beta \in K_\beta$;

\begin{tabular}{rlcl}
 \; approximate 
&$\{ (d_\alpha^0, e_\alpha^0 ), \ldots, (d_\alpha^t , \widehat e_\alpha ) \} $ 
&and
&$\{ (d_\beta^0, e_\beta^0 ), \ldots, (d_\beta^t , \widetilde e_\beta ) \} $ \\
by
&$\{ (d_\mu^0, e_\mu^0 ), \ldots, (d_\mu^t , e_\mu^t ) \} $ 
&and
&$\{ (d_\nu^0, e_\nu^0 ), \ldots, (d_\nu^t , e_\nu^t ) \} $ 
\end{tabular}

\noindent
Fix $\mu,\nu$
such that $e_\mu^t \in U$ and $e_\nu^t \in V$ and
$\angle( d_\nu^i - d_\mu^i, e_\nu^i - e_\mu^i ) < \theta$
for all $i < t$
and $\angle( d_\nu^t - d_\mu^t, d_\beta^t - d_\alpha^t ) < \varepsilon/2$.
Then $\angle( d_\nu^t - d_\mu^t, e_\nu^t - e_\mu^t ) \le   
\angle( d_\beta^t - d_\alpha^t, y-x)   +
\angle ( d_\nu^t - d_\mu^t, d_\beta^t - d_\alpha^t )+
\angle(  e_\nu^t - e_\mu^t , y-x) < \theta$.

If $\angle( d_\beta^t - d_\alpha^t, x-y) \le \pi/2$, use
$\widetilde e_\alpha \in K_\alpha$ and $\widehat e_\beta \in K_\beta$;

\begin{tabular}{rlcl}
 \; approximate
&$\{ (d_\alpha^0, e_\alpha^0 ), \ldots, (d_\alpha^t , \widetilde e_\alpha ) \} $
&and
&$\{ (d_\beta^0, e_\beta^0 ), \ldots, (d_\beta^t , \widehat e_\beta ) \} $ \\
by
&$\{ (d_\mu^0, e_\mu^0 ), \ldots, (d_\mu^t , e_\mu^t ) \} $
&and
&$\{ (d_\nu^0, e_\nu^0 ), \ldots, (d_\nu^t , e_\nu^t ) \} $
\end{tabular}

\noindent
Fix $\mu,\nu$
such that $e_\nu^t \in U$ and $e_\mu^t \in V$ and
$\angle( d_\nu^i - d_\mu^i, e_\nu^i - e_\mu^i ) < \theta$
for all $i < t$
and $\angle( d_\nu^t - d_\mu^t, d_\beta^t - d_\alpha^t ) < \varepsilon/2$.
Then $\angle( d_\nu^t - d_\mu^t, e_\nu^t - e_\mu^t ) \le   
\angle( d_\beta^t - d_\alpha^t, x-y)   +
\angle ( d_\nu^t - d_\mu^t, d_\beta^t - d_\alpha^t )+
\angle(  e_\nu^t - e_\mu^t , x-y) < \theta$.
\end{proof}

Proposition \ref{prop-intro} is false when $\theta \le \pi/2$ and $n \ge 2$;
see Example \ref{ex-twistlimit}.
For an easy counter-example to the lemma in $\RRR^2$, for suitable $D,E$:
For $\alpha < \omega_1$, let $p_\alpha = \{ (d_\alpha, e_\alpha ) \}$, where
the $d_\alpha$ are distinct points on the $x$-axis and
the $e_\alpha$ are distinct points on the $y$-axis with
$\hgt(e_\alpha) = \hgt(d_\alpha) + 1$.
Then $\{ p_\alpha : \alpha < \omega_1 \}$ is an antichain in $\PPP^\theta_0$.

\section{On Absolute Continuity}
\label{sec-ac}
Here, we make some further remarks on absolute continuity
and give a proof of the $n = 1$ case of Theorem \ref{thm-main}.

Our forcing arguments
will obtain ``generic'' functions as limits of
absolutely continuous functions.  But such limits
are not in general absolutely continuous; for example,
in $\RRR$, \emph{every} continuous function on $[0,1]$
is a uniform limit of polynomials
(which are clearly absolutely continuous).
We shall prove absolute continuity by applying Lemma \ref{lemma-equi-ac}.

\begin{lemma}
If $f_j \to f$ pointwise, all $f_j$ are measurable functions,
$U \subseteq X$  is open,
and $\mu(f_j\iv(U)) \le \varepsilon$ for all $j$, 
then $\mu(f\iv(U)) \le \varepsilon$.
\end{lemma}
\begin{proof}
By pointwise convergence,
$f\iv(U) \subseteq \bigcup_{m \in \omega}  \bigcap_{j \ge m} f_j\iv(U) $.
\end{proof}

Applying this to $f\iv$:

\begin{lemma}
\label{lemma-equi-ac}
Assume that $f_j \in \HH(X)$ for all $j \in \omega$ and 
$f_j\iv \to f\iv$ pointwise, where $f \in \HH(X)$.
Assume also that for all
$\varepsilon > 0$ there is a $\delta > 0$ such that 
for all open $U$ and \emph{all} $j$,
$\mu(U)  < \delta \to \mu(f_j(U)) < \varepsilon$.
Then $f$ is absolutely continuous.
\end{lemma}

When $X = \RRR$, one way to obtain the hypotheses of this lemma is to 
bound uniformly the derivatives of the $f_j$.
For general $\RRR^n$, we use the Jacobian.
We review here some standard notation:

If $f: \RRR^n \to \RRR^n$, then $\partial_i f$ (where $i < n$)
denotes the partial derivative of $f$ with respect to the $i^{\mathrm{th}}$
variable.  Then $\partial_i f: \RRR^n \to \RRR^n$, assuming that this
derivative exists everywhere.
As usual $C^1(\RRR^n, \RRR^n)$ denotes the set of all
$f: \RRR^n \to \RRR^n$ such that each $\partial_i f$ exists everywhere
and is continuous.

As usual, $J_f$ denotes the Jacobian matrix; so 
$J_f : \RRR^n \to \RRR^{n^2} $, and $J_f(x)$ is an $n\times n$
matrix whose $j^{\mathrm{th}}$ column is $\partial_j f(x)$
(viewed as a column vector).
Recall that if $f$ and $f\iv$ are $C^1$ bijections,
then $J_{f\iv}(f(x)) = (J_f(x))\iv$.

Also, if $f$ is 1-1 and $C^1$ on $U$, then
$\mu(f(U)) = \int_U | \det J_f(x) |$.
Thus we could obtain the hypotheses of Lemma \ref{lemma-equi-ac}
if we had a uniform bound to all the $|\det J_{f_j}(x)|$.  However,
in our forcing argument, this turns out to be impossible for the
same reason that we cannot get $f$ and $f\iv$ to be Lipschitz in
Theorem \ref{thm-main}.
We shall get a somewhat weaker condition on $f$; 
$|\det J_{f}(x)| < 2$ will hold ``most of the time'',
that is,
$\mu (f( \{x : |\det J_{f}(x)| \ge 2 \} ) )$ will be finite.
We plan to apply Lemma \ref{lemma-bound-jac} below to each $f_j$.
We state it so that it applies both to
$C^1$ functions on $\RRR^n$ and 
to piecewise linear functions on $\RRR$.

\begin{definition}
\label{def-WZ}
For $f : \RRR^n \to \RRR^n$ and $\ell \in (0,\infty)$, let
$W^f_\ell = \{x \in \RRR^n : \ell \le  | \det J_f(x) |  \}$ and
$Z^f_\ell = \{x \in \RRR^n : \ell - 1  \le  |\det J_f(x) | \le \ell \}$.
\end{definition}

\begin{lemma}
\label{lemma-bound-jac}
Fix $f \in \HH(\RRR^n)$, and assume that $f$ is $C^1$
except on some finite set.
Assume also that 
$\int_{W^f_2} | \det J_f(x) | \, dx < \infty$.
Fix $\varepsilon > 0$.  Then choose $k \ge 2$ so that
$\int_{W^f_k} | \det J_f(x) | \, dx < \varepsilon/2$.
Let $\delta = \varepsilon/(2k)$.  Then
for all Borel sets $U$, $\mu(U)  < \delta \to \mu(f(U)) < \varepsilon$.
\end{lemma}
\begin{proof}
Let
$U = A \cup B$, where $A = U \backslash W^f_k$ and $B = U  \cap W^f_k$.
Then $\mu(f(A)) \le k \mu(A) < k\delta = \varepsilon/2$ and 
$\mu(f(B)) \le \mu(f( W^f_k)) =
\int_{W^f_k} | \det J_f(x) | \, dx < \varepsilon/2$,
so $\mu(f(A \cup B))  \le \varepsilon$.
\end{proof}

Our generic $f$ will not be differentiable, but it will be a limit
of functions $f_j$ to which Lemma \ref{lemma-bound-jac} will apply.
To make the lemma apply \emph{uniformly}, so that we can use
Lemma \ref{lemma-equi-ac}, we shall have a uniform bound 
$\Ups(\ell)$ to each $\mu( Z^{f_j}_\ell)$, and apply:

\begin{lemma}
\label{lemma-psi-bound}
Fix $f \in \HH(\RRR^n)$, and assume that $f$ is $C^1$
except on some finite set.  Then for all $k \ge 2$:
{\arraycolsep=1pt
\[
\begin{array}{rl}
\frac13 \sum_{\ell > k} \ell \mu(Z^f_\ell) \le
\mu(f(W^f_k)) =
\int_{W^f_k} | \det J_f(x)|\, dx& \le \\
\sum_{\ell > k}  \int_{Z^f_\ell} | \det J_f(x)|\, dx& \le
\sum_{\ell > k} \ell \mu(Z^f_\ell) \ \ .
\end{array}
\]
}
\end{lemma}
\begin{proof}
The ``$=$'' holds by the change-of-variables formula, the
second ``$\le$'' holds because $W^f_k = \bigcup_{\ell > k} Z^f_\ell$,
and the third ``$\le$'' holds because $ | \det J_f(x)| \le \ell$
for all $x \in Z^f_\ell$.
For the first ``$\le$'':  note that each point $x$ is in no more
than two different $Z^f_\ell$, and $ | \det J_f(x)| \ge \ell - 1$
for all $x \in Z^f_\ell$, so that
\[
\begin{array}{l}
\int_{W^f_k} | \det J_f(x)|\, dx \ge
\frac12 \sum_{\ell > k}  \int_{Z^f_\ell} | \det J_f(x)|\, dx \ge
\frac12 \sum_{\ell > k}  (\ell-1) \mu(Z^f_\ell) \ \ ,
\end{array}
\]
and now use $\frac12  (\ell-1) \ge \frac13 \ell$, which holds because
$\ell \ge k+1 \ge 3$.
\end{proof}

It might seem more elegant to let
$Z^f_\ell = \{x \in \RRR^n : \ell - 1  \le  |\det J_f(x) | < \ell \}$.
Then, the $Z^f_\ell$ would partition  $W^f_k$, and the
$\frac13$ in the lemma could be replaced by $\frac23$.
But, our forcing arguments (such as the proof of Lemma \ref{lemma-main-1})
will use the fact that since $[\ell -1 , \ell]$ is closed,
if $x \notin Z^f_\ell$, then also
$x \notin Z^g_\ell$ whenever the derivatives of $f,g$ are sufficiently
close to each other.

In the proof of Theorem \ref{thm-main}, we shall modify
the poset $\PPP^\theta_0$ to 
force an $f$ that is BAC.
To do this, each forcing condition $p$ will have a side condition
$\Ups_p$ that will enable us to apply
Lemma \ref{lemma-bound-jac} to $f$.
First, we describe the one-dimensional case,
where $\det J_f(x) $ is just $f'(x)$:

\begin{lemma}
\label{lemma-main-1}
Theorem \ref{thm-main} holds when $n = 1$.
\end{lemma}
\begin{proof}
As remarked in Section \ref{sec-basic},
it is enough to assume $\CH$ and construct a ccc poset
and prove that $V[G]$ contains the required $f$. 
Let $\PPP$ be the set of all pairs $p = (\sigma_p, \Ups_p)$ such that
\begin{itemizz}
\item[1.]
$\sigma_p \in \PPP_0$ and
$\Ups_p \in (\QQQ \cap (0,\infty) )  ^{< \omega}$;
let $m_p = \dom(\Ups_p) $.
\item[2.]
$\sum \{ \ell   \Ups_p(\ell) : \ell \ge 3 \ \&\ \ell < m_p \} < 1$.
\item[3.]
Whenever $3 \le \ell  < m_p$:
$\mu(Z^{h_{\sigma_p}}_\ell) < \Ups_p(\ell)$ and
$\mu(Z^{h_{\sigma_p\iv}}_\ell) < \Ups_p(\ell)$.
\item[4.]
$1/\max(2, m_p-1) < h'_\sigma(x) < \max(2, m_p-1)$ for all
$x \notin \dom(\sigma)$.
\end{itemizz}
In (3), $h_\sigma$ is as defined in the proof of Example \ref{ex-not-ac}.
Define $q \le p$ iff $\sigma_q \le \sigma_p$ and $\Ups_q \le \Ups_p$,
so $\one = (\emptyset,\emptyset)$.

Working in $V[G]$, let $f = \cl(\bigcup \{ \sigma_p : p \in G\})$;
$f(D) = E$ because
$\{p : d \in \dom(\sigma_p)\}$ and
$\{p : e \in \ran(\sigma_p)\}$ are dense whenever $d \in D$ and $e \in E$.
Let $\Ups = \bigcup \{ \Ups_p : p \in G\}$; $\dom(\Ups) = \omega$ because,
by Condition (4), the sets $\{p : m_p > \ell\}$ are dense.
Note that $\sum_{\ell \ge 3} \ell   \Ups(\ell)  \le 1 $.
We next prove that $f$ is AC (the proof for $f\iv$ is similar):

First note that for all $p \in G$ and all $\ell \ge 3$,
$\mu(Z^{h_{\sigma_p}}_\ell) < \Ups(\ell)$:  For $\ell < m_p$,
this is clear by (3), while for $\ell \ge m_p$,
$Z^{h_{\sigma_p}}_\ell = \emptyset$ by (4).

For $\varepsilon > 0$, choose $\delta = \delta_\varepsilon$ as follows:
choose $k \ge 2$ so that $\sum_{\ell > k}
\Ups(\ell) < \varepsilon/2$; then let $\delta = \varepsilon/(2k)$. 
Now, for $p = (\sigma_p, \Ups_p) \in G$, if $h = h_{\sigma_p}$
and $k \ge 2$:
$
\int_{W^h_k} h'(x)\, dx \le \sum_{\ell > k} \ell \mu(Z^h_\ell)
\le \sum_{\ell > k} \ell \Ups(\ell) 
$
by Lemma \ref{lemma-psi-bound}.  By Lemma \ref{lemma-bound-jac},
$\mu(U)  < \delta_\varepsilon \to \mu(h(U)) < \varepsilon$
for all Borel $U$.

Next, choose $p_j \in G$ for $j\in\omega$ such that 
$ h_{\sigma_{p_j}} \to f$
and
$ h\iv_{\sigma_{p_j}} \to f\iv$ pointwise.  To do this, choose $p_j$
so that $\dom(\sigma_{p_j})$  and $\ran(\sigma_{p_j})$ both meet the interval
$[a 2^{-j}, (a+1) 2^{-j}]$ for all $a \in \ZZZ \cap [-2^{2j}, 2^{2j}]$.
Then, $f$ is AC by Lemma \ref{lemma-equi-ac}.

Back in $V$, we need to prove that $\PPP$ is ccc, so
fix $p_\alpha \in \PPP$ for $\alpha < \omega_1$; we shall find
$\alpha \ne \beta$ with $p_\alpha \compat p_\beta $.
WLOG, each $p_\alpha = (\sigma_\alpha, \Ups)$,
with $m = \dom(\Ups) \ge 3$.  We may also assume that
each $|\sigma_\alpha| = t \ge 1$,
and $\sigma_\alpha = \{(d_\alpha^i, e_\alpha^i) : i < t\}$.
Further, we may assume that
$d_\alpha^i  < d_\beta^j$ and $e_\alpha^i  < e_\beta^j$ holds
whenever $i < j$ and $\alpha,\beta < \omega_1$.

Now, since $\PPP_0$ is ccc,
fix $\alpha \ne \beta$ with $\sigma_\alpha \compat \sigma_\beta $;
we shall get a $q = ( \sigma_\alpha \cup \sigma_\beta, \hat \Ups)$
such that $q \le p_\alpha$ and $q \le p_\beta$.
So $\hat m = \dom(\hat \Ups) \ge m$ and $\hat \Ups \supseteq \Ups$.
Taking $\hat \Ups = \Ups$ need not work because then
$q$ may fail to be in $\PPP$ because (3) or (4) could fail.
To partly handle (3),
we assume that there is some fixed rational $\varepsilon > 0$ such that
$\mu(Z^{h_{\sigma_\alpha}}_\ell) < \Ups(\ell) - \varepsilon$ and
$\mu(Z^{h_{\sigma_\alpha\iv}}_\ell) < \Ups(\ell) - \varepsilon$
holds for each $\alpha$ whenever $3 \le \ell  < m_p$,
and that  $\sum \{ \ell \Ups(\ell) :
\ell \ge 3 \ \&\ \ell < m \} < 1 -  \varepsilon$,
and that the $\sigma_\alpha$ are close enough together 
that for each $\alpha,\beta$,
$| d_\alpha^i -  d_\beta^i | < \varepsilon /(4 t)$ and
$| e_\alpha^i -  e_\beta^i | < \varepsilon /(4 t)$.
Furthermore, assume that for each $i $ with $i + 1 < t$,
and each integer $\ell$,
if the slope
$ ( e_\alpha^{i+ 1} - e_\alpha^i) / ( d_\alpha^{i+ 1} - d_\alpha^i) 
\notin [\ell - 1, \ell]   $
holds for some $\alpha$, then
$ ( e_\beta^{i+ 1} - e_\alpha^i) / ( d_\beta^{i+ 1} - d_\alpha^i) 
\notin [\ell - 1, \ell]   $
holds for all $\alpha,\beta$;
and, likewise,
for the slope of the inverse, 
$ ( d_\alpha^{i+ 1} - d_\alpha^i) / ( e_\alpha^{i+ 1} - e_\alpha^i) $.
This cures the problem with (3) for $\ell < m$.

However, (4) might fail
for $q$ because there is no way to bound, below or above,
the slope between a pair of points
$(d_\alpha^i, e_\alpha^i)$ and $(d_\beta^i, e_\beta^i)$.  
Let $\hat m$ be the smallest number $\ge m$ that makes (4) hold.
If $\hat m = m$, we are done.  Otherwise:

Let $\sigma = \sigma_\alpha \cup \sigma_\beta$.
When $m \le \ell < \hat m$, let $c_\ell = |C_\ell|$,
where $C_\ell = C_\ell^A \cup C_\ell^B$ and
\[
\begin{array}{l} 
C_\ell^A = \{i < t :\; (d_\alpha^i, e_\alpha^i) \ne (d_\beta^i, e_\beta^i)
\,\wedge\,
( e_\beta^{i} - e_\alpha^i) / ( d_\beta^{i} - d_\alpha^i)
\in [\ell-1, \ell] \} \\
C_\ell^B = \{i < t :\; (d_\alpha^i, e_\alpha^i) \ne (d_\beta^i, e_\beta^i)
\,\wedge\,
( d_\beta^{i} - d_\alpha^i) / ( e_\beta^{i} - e_\alpha^i)
\in [\ell-1, \ell] \} \ \ .
\end{array} 
\]
Let $\hat\Ups(\ell) =  (c_\ell \varepsilon)/ (2 t \ell)  $.
Note that $ C_\ell^A \cap C_k^B = \emptyset$, so no $i$ lies
in more than two of the $C_\ell$, so 
$\sum_{m \le \ell < \hat m  } c_\ell \le 2t$, and hence
$\sum_{m \le \ell < \hat m  } \ell  \hat\Ups(\ell) \le \varepsilon$,
which gives us (2); that is,
$\sum \{ \ell \hat\Ups(\ell) : \ell \ge 3 \ \&\ \ell < \hat m \} < 1 $.
To verify (3) when $m \le \ell < \hat m$, note that,
using $| e_\alpha^i -  e_\beta^i | < \varepsilon /(4 t)$:\;
$\mu(Z^{h_{\sigma}}_\ell)  \le 
\sum_{i \in C_\ell^A} | d_\beta^i - d_\alpha^i| \le
c_\ell \cdot  \varepsilon /(4 t (\ell - 1))  <
c_\ell \cdot  \varepsilon /(2 t \ell) $;
to bound $\mu(Z^{h_{\sigma\iv}}_\ell)$, use $C_\ell^B$.
\end{proof}

In the higher dimensional case, we have no natural analog of 
$h_\sigma$; instead, our side conditions will include a function chosen
from $\FF_\theta$, defined below.   First,
a remark on norms; we use the Pythagorean norm on vectors in $\RRR^n$
and the operator norm on matrices:

\begin{definition}
\label{def-norms}
For $\vec v \in \RRR^n$, let
$\| \vec v \| =  (\sum_{i < n} (v_i)^2) ^ {1/2}$, and
when $Y$ is an $n\times n$ matrix,
let $\|Y\| =\sup \{ \| Y \vec v \| : \vec v \in S^{n-1} \}$.
\end{definition}

\begin{definition}
\label{def-FF}
When $\theta>0$, let $\FF_\theta = \FF^n_\theta$
denote the set of all $f$ such that:
\begin{itemizz}
\item[1.] $f$ is a bijection from $\RRR^n$ onto $\RRR^n$.
\item[2.] $f$ and $f\iv$ are $C^1$.
\item[3.] $\exists r \;
\exists \vec c \;
\forall \vec x\;
[\|\vec x\| \ge r \to f(\vec x) =  \vec c + \vec x]$.
\item[4.] $\tw(f) < \theta$.
\end{itemizz}
\end{definition}

Applying (2)(3),

\begin{lemma}
If $f \in \FF_\theta$,
then $f\iv \in \FF_\theta$, and $f$ and $f\iv$ are BAC.
\end{lemma}

We remark that replacing ``bijection'' by ``injection'' in (1)
results in an equivalent definition:

\begin{lemma}
\label{lemma-inj-surj}
Assume that $f: \RRR^n \to \RRR^n$ is 1-1 and continuous and satisfies $(3)$
above.  Then $f$ is a bijection.
\end{lemma}
\begin{proof}
If $n = 1$, this is obvious by the Intermediate Value Theorem,
so assume that $n > 1$.  Now, assume that $\vec d \notin \ran(f)$.
Replacing $f$ by $\vec x \mapsto x - \vec d$, we may assume
that $\vec 0 \notin \ran(f)$.

Define $\rho(\vec y) = \vec y / \|\vec y\|$,
so $\rho$ is the natural retraction of 
$\RRR^n \backslash \{0\}$ onto  $S^{n-1}$.
For $t \in [0, \infty)$, define $h_t : S^{n-1} \to S^{n-1}$ by
$h_t(\vec v) = \rho(f(t \vec v))$.  Then $h_0$ is the constant map
$\vec v \mapsto \rho(f(\vec 0))$.
Fix $r, \vec c$ as in (3).  For $t \gg \max(r, \| \vec c \|)$,
$h_t(\vec v) = ( \vec c + t \vec v) / \| \vec c + t \vec v \| \approx
t \vec v / t = \vec v$, so $h_t$ converges uniformly to the identity
map as $t \to \infty$.  But then, the identity map on $S^n$ 
is homotopic to a constant map, which is impossible.
\end{proof}

Another simple remark:

\begin{lemma}
If $f \in \FF_\theta$, then $\det J_f(\vec x)  > 0$ for all $\vec x$.
\end{lemma}
\begin{proof}
$\det J_f(\vec x)  \ne 0$ for all $\vec x$ by (2), and
$\det J_f(\vec x)  = 1$ for large enough $\vec x$ by (3), so use the
fact that $\RRR^n$ is connected.
\end{proof}

Some more notation on norms:

\begin{definition}
\label{def-dist}
For $f: \RRR^n \to \RRR^n$, 
$\|f\| = \sup\{ \|f(x)\| : x \in \RRR^n \}$, and
$\|J_f\| = \sup\{ \|J_f(x)\| : x \in \RRR^n \}$.

For $f,g \in \FF_\theta$, let
$d(f,g) =
 \max( \| f - g \| , \| f\iv  - g\iv \| ) $.
Then, the ball $B(f,\varepsilon) = \{g \in \FF_\theta :
d( f , g ) < \varepsilon\}$.
\end{definition}

Of course $\|f\|$ and/or $\|J_f\|$ may be $\infty$,
and $\|J_f\|$ is only defined when $f$ is differentiable.
When $f,g \in \FF_\theta$, $\|f\| = \infty $, but
$d( f , g ) < \infty$ and $\|J_f\| < \infty$.

For forcing, it will be convenient to use 
the distance function $d$,
since it preserves the symmetry between $f$ and $f\iv$:

\begin{definition}
\label{def-poset}
Following the terminology of Definition \ref{def-poset-approx},
and assuming $\CH$,
let $\PPP^\theta$ be $\one$ together with the set of all 
quadruples $p = (\sigma_p, h_p, \varkappa_p, \Ups_p)$ such that:

\begin{itemizz}
\item[1.]
$\sigma_p \in \PPP^\theta_0$ and
$\Ups_p \in (\QQQ \cap (0,\infty) )  ^{< \omega}$;
let $m_p = \dom(\Ups_p) $.
\item[2.]
$\sum \{ \ell   \Ups_p(\ell) : \ell \ge 3 \ \&\ \ell < m_p \} < 1$.
\item[3.] $ h_p \in \FF_\theta$ and $h_p \supseteq \sigma_p$.
\item[4.] $ \varkappa_p$ is a positive rational number.
\item[5.]
Whenever $3 \le \ell  < m_p$:
$\mu(Z^{h_{p}}_\ell) < \Ups_p(\ell)$ and
$\mu(Z^{h_{p}\iv}_\ell) < \Ups_p(\ell)$.
\item[6.]
$1/\max(2, m_p -1) <   \det J_{h_p}(x) < \max(2, m_p -1)$ for all $x$.
\end{itemizz}
Define $q \le p$ iff $p = \one$ or $p,q$ are quadruples with 
$\sigma_q \supseteq \sigma_p$
and
$\Ups_q \supseteq \Ups_p$
and
$\varkappa_q \le \varkappa_p$
and
$B(h_q, \varkappa_q) \subseteq B(h_p, \varkappa_p)$.
\end{definition}

So, $h_p$ is an approximation to the $f$ that we are constructing,
and $\varkappa_p$ is a ``promise'' that this $f$ will
satisfy $d(f, h_p) \le \varkappa_p$.
There is no natural $\one$ in this poset, so we added one artificially,
on top of all the ``natural'' forcing conditions.
Note that $(\sigma, h, \varkappa', \Ups) \le (\sigma, h, \varkappa, \Ups)$
always holds whenever $  \varkappa' \le \varkappa$. 
Also, by (6):

\begin{lemma}
\label{lemma-dense-m}
$\{p : m_p > \ell\}$ is dense for each $\ell$.
\end{lemma}

Also, we note that we can make a ``small change'' in $h_p$ and obtain
an extension of $p$:

\begin{lemma}
\label{lemma-smallchange}
For each $p = (\sigma, h, \varkappa_p, \Ups_p) \in \PPP^\theta$,
there is a rational $\zeta = \zeta_p > 0$ such that for all $g \in \FF_\theta$:

\emph{If} $d(g,h) < \varkappa_p$, and $g \supseteq \sigma$, and
$\mu(\overline S), \mu(\overline T) \le \zeta$, where
$S = \{x : g(x) \ne h(x)\}$ and
$T =  \{y : g\iv(y) \ne h\iv(y)\}$, \emph{then}
there is a $q \le p$ of the form $q = (\sigma, g, \varkappa_q, \Ups_q) $.
\end{lemma}
\begin{proof}
Choose $\zeta $ so that:
(A)  $\zeta < \Ups_p(\ell) - \mu(Z^{h}_\ell) $ and
$\zeta < \Ups_p(\ell) - \mu(Z^{h\iv}_\ell)$ for all $\ell < m_p$,
and
(B) $4 \zeta < 1 -
\sum \{ \ell   \Ups_p(\ell) : \ell \ge 3 \ \&\ \ell < m_p \}$.

For $q \le p$:  We need $\varkappa_q \le \varkappa_p$
and $B(g, \varkappa_q) \subseteq B(h, \varkappa_p)$,
and these are satisfied if we just choose
$\varkappa_q < \varkappa_p - d(g,h)$.

But we also need
$\Ups_q \supseteq \Ups_p$ (so $m_q \ge m_p$),
and we must be careful to define $q$ to satisfy $(1 - 6)$.
For (6), choose any $m_q \ge \max(2, m_p)$ such that
$1/ (m_q -1) <   \det J_{g}(x) < ( m_q -1)$ for all $x$.

For (5): (A) implies
that (5) (for $\ell < m_p$) continues to hold with $g$ replacing $h$.
If $m_q = m_p$, we are now done, so assume that $m_q > m_p$.
Also, assume that $m_q \ge 4$, since otherwise (5) and (2) are vacuous.

To ensure (5) when $\max(3, m_p) \le \ell < m_q$: choose rational
$\Ups_q(\ell)$ such that
$ \mu(Z^{g}_\ell) + \mu(Z^{g\iv}_\ell) < \Ups_q(\ell) <
 \mu(Z^{g}_\ell) + \mu(Z^{g\iv}_\ell)  + \zeta / m_q$.
But now for (2):  We've added
$\sum\{\ell \Ups_q(\ell) : \max(3, m_p) \le \ell < m_q \}$ 
to the $\sum$ in (2).
This amount is bounded above by $\zeta$ (from the $\zeta / m_q$ terms) plus
\[
\sum_{\ell > k} [ \ell \mu(Z^{g}_\ell) + \ell \mu(Z^{g\iv}_\ell)  ] \le
3 [ \mu(g( W^{g}_k)) + \mu(g\iv( W^{g\iv}_k)) ] \le
3 [ \mu(\overline T) + \mu(\overline S) ] \le 3 \zeta \, ,
\]
where $k = \max(3, m_p) - 1$ (see Lemma \ref{lemma-psi-bound}),
so we are done by (B).

To verify the second ``$\le$'' above, use
$g( W^{g}_k) \subseteq \overline T$ and
$g\iv( W^{g\iv}_k) \subseteq \overline S$.
To verify 
$g( W^{g}_k) \subseteq \overline T$, fix $x \in W^{g}_k$.
Then $\det J_g(x) \ge \max(3, m_p) - 1$.
But also $\det J_{h}(x) < \max(2, m_p -1)$,
so $J_g(x) \ne J_h(x)$, and hence $x \in \cl(S)$
so $g(x) \in \cl(g(S))$; but $g(S) = h(S) = T$ because
$g$ and $h$ are bijections.
\end{proof}

We now need the following two lemmas,
whose proofs are a bit more complex than the corresponding results
used in the proof of Lemma \ref{lemma-main-1}:

\begin{lemma}
\label{lemma-dense-de}
For $\vec d \in D$ and $\vec e \in E$, both sets
$\{p :  \vec d \in \dom(\sigma_p) \}$ and
$\{p :  \vec e \in \ran(\sigma_p) \}$  are dense in $\PPP^\theta$.
\end{lemma}

\begin{lemma}
\label{lemma-ccc}
$\PPP^\theta$ is ccc whenever $\theta > \pi/2$.
\end{lemma}

These lemmas will be proved in Sections \ref{sec-mat} and \ref{sec-ccc}, after
we prove some more facts about twists and Jacobians.

\begin{proofof}{Theorem \ref{thm-main}}
As in the proof of Lemma \ref{lemma-main-1}, it is enough to
assume $\CH$, construct $\PPP^\theta$ (which is ccc by Lemma \ref{lemma-ccc}),
and show that $V[G]$ contains the required $f$.  We again have
$f = \cl(\bigcup \{ \sigma_p : p \in G\})$
and $\Ups = \bigcup \{ \Ups_p : p \in G\}$.
Since $f$ and $f\iv$ are uniform limits of continuous bijections,
$f$ is a continuous bijection of $\RRR^n$ onto $\RRR^n$.
$\tw(f) \le \theta$ by Lemma \ref{lemma-cl-tw}.
Also, $f(D) = E$ by Lemma \ref{lemma-dense-de},
and absolute continuity for $f$ and $f\iv$
is proved as in Lemma \ref{lemma-main-1}.
\end{proofof}

\section{Twists and Jacobians}
\label{sec-mat}
\begin{definition}
\label{def-nice}
$p = (\sigma, h, \varkappa,\Ups) \in \PPP^\theta$ is \emph{nice}
iff for all $(\vec d, \vec e) \in \sigma$,
$h(\vec x) =  \vec x + \vec e - \vec d$
holds in some neighborhood of $\vec d$.
\end{definition}

\begin{lemma}
\label{lemma-nice-are-dense}
The set of all nice $p$ is dense in $\PPP^\theta$.
\end{lemma}

This will be used in the proof of ccc (Lemma \ref{lemma-ccc}).
That proof will use the same basic idea as the ccc proof from
Lemma \ref{lemma-main-1}, which relied on establishing
``$\sigma_\alpha \compat \sigma_\beta \to p_\alpha \compat p_\beta $''.
In the proof of Lemma \ref{lemma-ccc}, we can now say
WLOG that all the $p_\alpha$ are nice.  The fact that $h_\alpha$
and $h_\beta$ are just translations near the various
$(\vec d, \vec e) \in \sigma_\alpha \cup \sigma_\beta$
will aid in the proof of $p_\alpha \compat p_\beta$.

We shall prove Lemma \ref{lemma-nice-are-dense}
later in this section, after some preliminaries.

Because we are using the operator norm on the Jacobian,
there is a Lipschitz condition in terms of $\|J_f\|$ when
$\|J_f\| < \infty$:

\begin{lemma}
\label{lemma-lip}
If $f \in C^1(\RRR^n, \RRR^n)$ then
$\|f(c) - f(a)\| \le \|J_f\|\,  \|c - a\|$ for all $c,a \in \RRR^n$.
\end{lemma}
\begin{proof}
Let $b = c-a$.  
$\|f(c) - f(a)\| $ is no more than the length of the path
from $f(a)$ to $f(c)$ defined by $t \mapsto f(a + tb)$ for $t \in [0,1]$.
This length equals $\int_0^1 \|\frac{d}{dt} f(a + tb)\| \, dt =
\int_0^1 \|J_f(a + tb)\, b \| \, dt \le
\int_0^1 \| J_f \| \|b \| \, dt  = \| J_f \| \| b\|$.
\end{proof}

Using $J_f$, we can compute a ``local twist'':

\begin{definition}
If $Y$ is a non-singular matrix, let
\[
\twist(Y) = \{ \angle( \vec v , Y \vec v ) : \vec v \in S^{n-1} \}  =
\{ \angle( \vec v , Y \vec v ) : \vec v \in \RRR^n \backslash \{\vec 0\} \} 
\ \ .
\]
Then, let $\tw(Y) = \sup(\twist(Y)) \in [0,\pi]$.
\end{definition}

Observe that for
$f \in \FF_\theta$,
$\tw(J_f(x)) < \theta $ for all $x$.
Also, note that $\twist(Y) = \twist(Y\iv)$.
Also, if $f$ is the function $\vec v \mapsto Y \vec v$,
then $\twist(Y) = \twist(f)$ and $\tw(Y) = \tw(f)$ .

Next, a remark on elementary geometry.
Let $v$ be the center of the Earth and $x$ a point on its surface,
and
let $w$ be the center of the Moon and $y$ a point on its surface.
Then the lines $\overrightarrow{vw}$ and 
$\overrightarrow{xy}$ point in ``almost'' the same direction,
and the following lemma gives a crude upper bound to the angle between them:

\begin{lemma}
\label{lemma-ball}
In $\RRR^n$:  say 
$\| w - v \| = T$ \textup(the \emph{distance}\textup),
and $\|x - v\| =  r$ and $\|y - w\| =  s$
\textup(the \emph{two radii}\textup), and assume
that $T \ge r + s$.  Let $\beta = \angle( w - v, y - x)$.
Then $\beta \le \pi(r+s) / (2T)$.
\end{lemma}
\begin{proof}
$\beta = \angle( w - v, (y + v - x) - v)$.  Consider 
$\triangle ABC$, where $A,B,C$ are the points 
$y + v - x$,  $v$, $w$, respectively.
Let $a,b$ be the lengths of the sides opposite $A,B$ respectively,
and let $\alpha$ be the angle at $A$; 
$\beta$ is the angle at $B$. 
Note that $b = \| y + v - x - w\| \le r + s \le T = a$.

By the ``law of sines'',
$b / \sin(\beta) = a / \sin(\alpha)$, so
$\sin(\beta) = (b/a) \sin(\alpha) \le b/a$.
Also, $\beta < \pi/2$ because $b \le a$, and 
$ 0 \le x \le \pi/2 \to \sin(x) \ge (2/\pi)\, x$,
so $\beta \le  (\pi/2)\,(b/a) \le  (\pi/2)\,( (r + s) /a )$.
\end{proof}

In many (but not all) of our applications,
one of $r,s$ will be $0$.
We remark that a precise upper bound is
$\beta \le \arcsin((r+s)  /T )$, but the one in the lemma
is simpler and will suffice in all our arguments.

We shall eventually prove the following, which is the 
``pure $\FF_\theta$'' analog of Lemma \ref{lemma-nice-are-dense}.

\begin{lemma}
\label{lemma-transl-near-0}
Assume that $f\in \FF_\theta$ and $f(\vec d) = \vec e$ and
$\varepsilon > 0$.  Then there
exists a $g \in \FF_\theta$ such that 
$d(f,g)  < \varepsilon$,
and
$g(\vec d) = \vec e$,
and
$g(\vec x) = f(\vec x)$ whenever $\|\vec x - \vec d\| \ge \varepsilon$
or $\|f(\vec x) - \vec e \| \ge \varepsilon$,
and 
$g(\vec x) = \vec x - \vec d + \vec e$ holds in some neighborhood of $\vec x$.
\end{lemma}

So, $g$ is close to $f$, but equals a simple translation near $\vec d$. 
A rough idea of the proof: 
By translating the domain and range,
we may assume that $\vec d = \vec e = \vec 0$;
then we need to get $g(\vec x) = \vec x$ for $\vec x$ near $\vec 0$.
We first modify $f$ slightly to get a function $h$
such that $h(\vec x) = A \vec x$ near $\vec 0$,
where $A = J_f(\vec 0)$. 
We then get $g$ by ``morphing'' $A$ to $I$ near $\vec 0$.
This ``morphing'' requires some further discussion of matrices:

\begin{definition}
For $n \ge 1$, $\MM^n$ denotes the space of all $n \times n$ real matrices;
this has the topology
of $\RRR^{n^2}$.  Then, for $\theta > 0$, define
$\NN^n_\theta = \{A \in \MM^n : \det A  > 0 \ \&\ \tw(A) < \theta\}$. 
\end{definition}

Some easy closure properties:

\begin{lemma}
\label{lemma-closure-NN}
$A \in   \NN^n_\theta \leftrightarrow A\iv \in \NN^n_\theta \leftrightarrow 
cA \in \NN^n_\theta \leftrightarrow  O\iv A O \in\NN^n_\theta$ whenever
$c > 0$ and $O$ is an orthogonal matrix.
\end{lemma}

$\NN^n_\theta$ is clearly open in $\MM^n$, and $I \in \NN^n_\theta$.  But:

\begin{question}
Is $\NN^n_\theta$ connected when $0 < \theta < \pi$?
\end{question}

The answer is trivially ``yes'' for $n = 1$.  It is also ``yes'' for $n = 2$,
as can be
proved by direct computation, using
Lemma \ref{lemma-closure-NN} to
simplify the form of the matrix. 
The following observation makes this question irrelevant 
for our work here:

\begin{lemma}
\label{lemma-path}
If $f \in \FF_\theta$ and $\vec a \in \RRR^n$ and $A = J_f(\vec a)$,
then $A \in \NN^n_\theta$ and there is a $C^\infty$ path
$\Gamma : [0,1] \to \NN^n_\theta$ such that $\Gamma(0) = I$
and $\Gamma(1) = A$.
\end{lemma}
\begin{proof}
To get a continuous $\Gamma$, fix
$r,\vec c$ as in (3) of Definition \ref{def-FF},
and then fix $\vec d$ with $\|\vec d\| > r$.
Then let $\Gamma(t) = J_f(t \vec a + (1 - t) \vec d)$.
Then, observe that (just because
$\NN^n_\theta$ is open in $\MM^n$), whenever $A,B$ lie in the same connected 
component of $\NN^n_\theta$, they are connected by 
a  $C^\infty$ path lying in $\NN^n_\theta$.
\end{proof}

The following lemma expresses the basic matrix morphing:

\begin{lemma}
\label{lemma-morph}
Assume that $h(\vec v) = A(\|\vec v\|) \vec v$, where 
$A : [0,\infty) \to \MM^n$ and for each $r \in [0,\infty)$,
$A(r)$ is non-singular and $\tw(A(r)) < \theta$.
Assume that $M := \sup\{ \| A(r)\iv \| : r \in [0,\infty) \} < \infty$.
Fix $\varepsilon \in (0 , \pi / 2)$ and assume that:
\[
\| A( (1 + \sigma) r )  - A(r) \| < (\varepsilon \sigma) / (\pi M)
\qquad (\forall \sigma, r > 0) \ \ .  
\tag{$\ast$}
\]
Then $h$ is 1-1 and $\tw(h) < \theta + \varepsilon/2$.
Furthermore,
\[
\|h(\vec v_1) - h(\vec v_0)\| \ge \|\vec v_1 - \vec v_0\| / (2 M) \tag{\dag}
\]
for all $\vec v_0, \vec v_1$.
\end{lemma}
\begin{proof}
First, we establish (\dag), which implies that $h$ is 1-1.
Let $A_i = A(\vec v_i)$ for $i = 0,1$.  Observe:
\[
\begin{array}{ll}
\|A_1 (\vec v_1 - \vec v_0)\| \ge
\|\vec v_1 - \vec v_0\| / \| A_1\iv\| \ge
\|\vec v_1 - \vec v_0\| / M  \qquad\qquad\qquad\qquad     &(1)  
\end{array}
\]
Since $\dag$ is clear from (1) when $\|\vec v_1\| = \| \vec v_0\|$
or $\vec v_0 = \vec 0$,
we may assume that $\| \vec v_0 \| = r$ and 
$\| \vec v_1 \| = (1 + \sigma)r$, where $\sigma, r > 0$.
Then (\dag) follows using (2)(1)(3):
\[
\begin{array}{ll}
\|h(\vec v_1) - h(\vec v_0)\| = 
\|A_1 \vec v_1 - A_0 \vec v_0\| = 
\|A_1 (\vec v_1 - \vec v_0) + (A_1 - A_0) \vec v_0\| &(2) \\
\|(A_1 - A_0) \vec v_0\| \le
(\varepsilon \sigma r) / (\pi M) \le
\|\vec v_1 - \vec v_0\| \varepsilon  / (\pi M) \le
\|\vec v_1 - \vec v_0\|  / (2M) &(3)  
\end{array}
\]

For $\tw(h) \ < \theta + \varepsilon/2$, we
must show that $\angle( \vec v_1 - \vec v_0, h(\vec v_1) - h(\vec v_0) )
< \theta + \varepsilon/2 $ whenever $\vec v_1 \ne \vec v_0$.
This is clear if $\| \vec v_1\|  = \| \vec v_0\| $ or
if one of $\vec v_1 , \vec v_0$ is $\vec 0$, so we may assume that
$\vec v_0, \vec v_1, A_0, A_1, r, \sigma$ are as above,
and we must show that
\[
\angle( \vec v_1 - \vec v_0, A_1 \vec v_1  - A_0 \vec v_0) 
< \theta + \varepsilon/2
\]
Now, using $\tw(A(r)) < \theta$, we know that
$ \angle( \vec v_1 - \vec v_0, A_1 \vec v_1  - A_1 \vec v_0) < \theta $,
so we now use Lemma \ref{lemma-ball} to show that
\[
\beta := 
\angle(  A_1 \vec v_1  - A_1 \vec v_0  ,  A_1 \vec v_1  - A_0 \vec v_0)
\le \varepsilon/2 \ \ .
\]
The ``distance'' is
$T = \| A_1 \vec v_1  - A_1 \vec v_0 \| \ge
\| \vec v_1  -  \vec v_0\| / M \ge \sigma r / M$, using (1),
and the two ``radii'' are $0$ and
$\|  A_1 \vec v_0 - A_0 \vec v_0 \| \le r \cdot (\varepsilon \sigma) / (\pi M) $
by $(\ast)$, so that $\beta \le 
\pi \cdot   r \cdot (\varepsilon \sigma) / (\pi M) \div
2 \sigma r / M  = \varepsilon / 2 $.
\end{proof}

We shall obtain the $A(r)$ using a path in 
$\NN^n_\theta$, with the aid of the following:

\begin{lemma}
\label{lemma-smooth-log}
Given $P,Q, \zeta > 0$, with $P < Q e^{-1/\zeta}$, there is 
a non-decreasing $C^\infty$ function $\varphi: \RRR \to [0,1]$ such that
$\varphi(x) = 0$ whenever $x \le P$, and
$\varphi(x) = 1$ whenever $x \ge Q$, and
$\varphi( (1 + \sigma) x) - \varphi(x) \le \zeta \sigma$
whenever $\sigma, x > 0$.
\end{lemma}
\begin{proof}
Fix $P', Q', \zeta'$ such that
$P < P' < Q' = e^{1/\zeta'} P' < Q$
and $0 < \zeta' < \zeta$.
Now, let $\psi(x)$ be 
$0$ when $x \le P' $, $1$ when $x \ge Q'$, and
$\zeta' \log(x/P')$ when $P' \le x \le Q'$.
Then
$\psi( (1 + \sigma) x) - \varphi(x) \le \zeta' \sigma$
whenever $\sigma, x > 0$, but $\psi$ does not
satisfy the lemma because, although it is continuous, it is not $C^1$.

To obtain a $C^\infty$ function,
fix $a > 0$ such that $a < Q - Q'$ and $a < P' - P$ and
$a < (\zeta -\zeta') P / (1 + \zeta)$,
and convolve $\psi$ with a smooth function supported on $[-a,a]$.
Let $\delta: \RRR \to [0,1]$ be a $C^\infty$ function such
that $\delta(t) = 0$ whenever $|t| \ge a$ and 
$\delta(t) = \delta(-t)$ for all $t$ and
$\int^\infty_{-\infty} \delta(t) \, dt = 1$.
Then let 
\[
\varphi(x) = 
\int^\infty_{-\infty} \delta(t) \psi(x - t) \, dt = 
\int^\infty_{-\infty} \delta(x - u) \psi(u) \, du \ \ .
\]
Then $\varphi$ satisfies everything required except possibly
for $\varphi( (1 + \sigma) x) - \varphi(x) \le \zeta \sigma$
whenever $\sigma, x > 0$.  Rewrite this as the equivalent
\[
0 < x < y \to
\varphi(y) - \varphi(x)  \le \zeta (y - x) / x  \ \ .\tag{$\ast$}
\]
This is clear when $y \le P$ (since then $\varphi(y) - \varphi(x) = 0$),
so assume always that $y > P$.  
Also, $(\ast)$ is clear when $\zeta (y - x) / x \ge 1$, which
is equivalent to $\zeta y \ge (1 + \zeta) x$.
Using $y > P$, we may assume now also that 
$\zeta P < (1 + \zeta) x$.
This implies that $x - a > 0$ (using our third assumption on $a$),
which justifies the following, using
$0 < u < v \to \psi(v) - \psi(u)  \le \zeta' (v - u) / u $:
\[
\varphi(y) - \varphi(x)  =
\int^\infty_{-\infty} \delta(t) [\psi(y - t) - \psi(x - t)] \, dt \le
\zeta' \int^\infty_{-\infty} \delta(t) [(y - x) / (x - t)] \, dt  \ \ .
\]
This will give us $(\ast)$ \emph{if} we know that
\[
\forall t \in [-a, a] \; \;
\big( \zeta' [(y - x) / (x - t)] \le \zeta (y - x) / x \big) \ \ .
\tag{\dag}
\]
But $(\dag)$ is equivalent to
$\zeta' / \zeta \le   \min\{ (x - t ) / x : t \in [-a,a]\}$,
and this $\min$ is just $1 - a/x $, so we shall have $(\dag)$
if $a/x \le 1 - \zeta' / \zeta = (\zeta - \zeta')/\zeta$.
Since we are assuming that $x > \zeta P / (1 + \zeta) $,
we just need $a \le (\zeta - \zeta') P / (1 + \zeta) $,
which follows from our third assumption on $a$.
\end{proof}

\begin{lemma}
\label{lemma-transl-special}
Lemma \ref{lemma-transl-near-0} holds in the special case that
$\vec d = \vec e = \vec 0$ and 
$f(\vec x) = A \vec x$ in some neighborhood of $\vec 0$.
\end{lemma}
\begin{proof}
Fix $\hat \theta \in (0, \theta)$ such that $f \in \FF_{\hat \theta}$;
make sure that $\theta - \hat \theta < \pi/2$.
Then, applying Lemma \ref{lemma-path}, let
$\Gamma: [0,1] \to \NN^n_{\hat\theta}$ be a $C^\infty$
path in $\NN^n_{\hat\theta}$ with $\Gamma(0) = I$ and $\Gamma(1) = A$.
Note that a smooth path is also Lipschitz, so fix $K > 0$
such that $\|\Gamma(t_0) - \Gamma(t_1)\| \le K |t_0 - t_1|$
for all $t_0,t_1 \in [0,1]$.  Also fix $R > 0$ such that
$f(\vec v) = A \vec v$ whenever $\|\vec v \| \le R$.
Let $M = \sup\{ \| (\Gamma(t)) \iv \| : t  \in [0,1] \}$.
Let $C = \inf\{ \|f(\vec v)  \| : \|\vec v\| \ge R \}$.
Let $J = \sup\{ \| \Gamma(t) \| :  t \in [0,1] \}$.
Then, choose $Q, \zeta$ satisfying:
\begin{itemizz}
\item[a.] $0 < \zeta <  (\theta - \hat \theta) / (\pi M K) \le 1/(2KM)$.
\item[b.] $0 < Q < R$ 
\item[c.] $J Q < C$ and
$JQ / ( C - \|A\| Q) < (\theta - \hat\theta)/\pi$.
\item[d.] $Q < \varepsilon/2$.
\item[e.] $JQ  < \varepsilon/2$ and $\forall \vec x \,
[ \| \vec x \| \le Q \to \| f(\vec x) \| < \varepsilon/2 ]$.
\end{itemizz}

Fix $P \in (0, Q e^{-1/\zeta})$, and then 
fix $\varphi$ as in Lemma \ref{lemma-smooth-log}.
Let $A(r) = \Gamma(\varphi(r))$.  Then $A(r) = I$ for $r \le P$
and $A(r) = A$ for $r \ge Q$.
Define $h(\vec v) = A( \| \vec v \| ) \vec v$.
By Lemma \ref{lemma-morph},
$\tw(h) < \theta$ and $h$ is 1-1 if we can show:
\[
\| \Gamma( \varphi( ( 1 + \sigma) r)) - \Gamma(\varphi(r)) \| <
\frac{  (\theta - \hat \theta) } {\pi  M} \sigma
\qquad (\forall \sigma, r > 0) \ \ .  
\]
But this follows from (a) above, using
$\varphi( (1 + \sigma) r) - \varphi(r) \le \zeta \sigma$ and 
our Lipschitz constant $K$, which implies that
$ \| \Gamma( \varphi( ( 1 + \sigma) r)) - \Gamma(\varphi(r)) \| \le
\zeta K \sigma$.

Note that $h(\vec v) = f(\vec v)$ whenever $Q \le \vec v \le R$.
Let $g(\vec v)$ be
$h(\vec v)$ when $\| \vec v \| \le R$ and
$f(\vec v)$ when $\| \vec v \| \ge Q$.

To show that $g$ is 1-1:  
fix $v_0, v_1$ with $v_0 \ne v_1$; we must show that $g(v_0) \ne g(v_1)$.
Let $r_i = \| \vec v_i\|$.
We may assume that $r_0 \le r_1$.  But also,
$g(v_0) \ne g(v_1)$ is clear whenever
$g \res \{ \vec v_0 , \vec v_1 \}$ equals either
$f \res \{ \vec v_0 , \vec v_1 \}$ or
$h \res \{ \vec v_0 , \vec v_1 \}$, so we may assume that
$r_0 < Q$ and $r_1 > R$.
Then $\|g(v_0)\| = \| A(r_0) v_0 \| \le JQ$ and
$\|g(v_1)\| = \| f(v_1) \| \ge C$, so $g(v_0) \ne g(v_1)$ because $JQ < C$.

To prove that $\tw(g) < \theta$, 
fix $v_0, v_1, r_0, r_1$ as above with
$v_0 \ne v_1$; we must show that 
that $\angle(\vec v_1 - \vec v_0, g(\vec v_1) - g(\vec v_0) ) < \theta$.
By the same reasoning, we may assume that $r_0 < Q$ and $r_1 > R$. 

Now, we have
$\angle(\vec v_1 - \vec v_0, f(\vec v_1) - f(\vec v_0) ) < \hat\theta$,
and shall use Lemma \ref{lemma-ball} to
conclude that $\angle(\vec v_1 - \vec v_0, g(\vec v_1) - g(\vec v_0) )$
by verifying that
\[
\beta := 
\angle( f(\vec v_1) - f(\vec v_0) , g(\vec v_1) - g(\vec v_0))  \le
\theta - \hat\theta  \ \ .
\]
Note that $ g(\vec v_1) = f(\vec v_1) $, while
$ g(\vec v_0) = h(\vec v_0) = A(r_0) \vec v_0$ and
$ f(\vec v_0) = A \vec v_0$.
Then the ``distance'' is $T = \|  f(\vec v_1) -  f(\vec v_0) \| \ge
C - \|A\| Q$,
and the two ``radii'' are
$\| f(\vec v_1) - g(\vec v_1) \| = 0$ and
$\| f(\vec v_0) - g(\vec v_0) \| =  \| (A - A(r_0) ) \vec v_0 \| \le 2JQ$,
so $\beta \le \pi \cdot  2JQ \div 2( C - \|A\| Q ) \le \theta - \hat\theta $
by (c).

To prove that 
$g(\vec x) = f(\vec x)$ whenever $\|\vec x \| \ge \varepsilon$
or $\|f(\vec x)  \| \ge \varepsilon$:
For $\|\vec x \| \ge \varepsilon$, just use $Q < \varepsilon$, by (d).
For $\|f(\vec x)  \| \ge \varepsilon$, use (e), which implies
that $\|f(\vec x)  \| \ge \varepsilon \to \|\vec x\| \ge Q 
\to g(\vec x) = f(\vec x)$.

To prove that $\| g - f \| < \varepsilon$, use (e) to show
that $\|\vec x \| \le Q$ implies that
$\| g(x) - f(x)\| \le \| g(\vec x) \| + \| f(\vec x) \|
\le JQ +  \| f(\vec x) \| < \varepsilon/2 + \varepsilon/2 $.

To prove that $\| g\iv - f\iv \| < \varepsilon$:
We want $f(\vec x) = g(\vec z) \to \| \vec x - \vec z \| < \varepsilon$.
Since $f, g$ are both 1-1, this is trivial unless
$f(\vec x) \ne g(\vec x)$ and $f(\vec z) \ne g(\vec z)$.
Then $\|\vec x\|, \|\vec z\| < Q$, so apply the fact that $Q < \varepsilon / 2$.

Finally, we must prove that
$g\iv$ is $C^1$.  Since $f\iv$ is $C^1$, it is sufficient to prove that
$h\iv$ is $C^1$.
Since $h$ is a $C^1$ bijection, it is sufficient to prove
that $J_h$  is everywhere non-singular, which follows if
we show that $h\iv$ is Lipschitz; but this is clear from
Lemma \ref{lemma-morph}.
\end{proof}

Next, we need to show that every function in $\FF_\theta$
is close to some $f \in \FF_\theta$ 
such that $f(\vec x) = A \vec x$ in some neighborhood of $\vec 0$.
We first show that every ``small modification'' of a function in 
$\FF_\theta$ also lies in $\FF_\theta$.

\begin{lemma}
\label{lemma-small-change}
Fix $f \in \FF_\theta$, and fix $\hat \theta \in (\tw(f), \theta)$ 
with $\theta - \hat \theta < \pi/2$.
Let $g: \RRR^n \to \RRR^n$ be a $C^1$ function such that 
$\exists r \; \forall \vec x\; [\|\vec x\| \ge r \to g(\vec x) =  \vec 0]$.
Assume also
\[
\| g(\vec v_1) - g(\vec v_0) \| \le \frac  2 \pi
(\theta -  \hat \theta) \| f(\vec v_1) - f(\vec v_0) \|
\qquad (\forall \vec v_0, \vec v_1 \in \RRR^n) \ \ .  
\tag{\ding{"50}}
\]
Then $f + g \in \FF_\theta$.
Furthermore, $d(f, f+g) \le \|g\| \cdot \max(1, \| J_{f\iv} \|)$.
\end{lemma}
\begin{proof}
Let $h = f+g$.
It is clear that $h$ is $C^1$ and 1-1 and satisfies (3)
of Definition \ref{def-FF}.
It follows from Lemma \ref{lemma-inj-surj} that $h$ is a bijection.
It is easy to see from (\ding{"50}) that $J_h(\vec x)$ must be 
non-singular, so that $h\iv$ is also $C^1$.

To prove that $\tw(h) < \theta$,
we must show that 
$\angle(\vec v_1 - \vec v_0, h(\vec v_1) - h(\vec v_0) ) < \theta$.
whenever $\vec v_0 \ne \vec v_1$.  Now
$\angle(\vec v_1 - \vec v_0, f(\vec v_1) - f(\vec v_0) ) < \hat\theta$,
so we apply Lemma \ref{lemma-ball} to show that
\[
\beta := 
\angle( f(\vec v_1) - f(\vec v_0) , h(\vec v_1) - h(\vec v_0))  \le
\theta - \hat\theta  \ \ .
\]
Here, $h = f+g$, so $\beta =  \angle( f(\vec v_1) - f(\vec v_0) ,
 f(\vec v_1)  - [ f(\vec v_0) +  g(\vec v_0) -  g(\vec v_1) ] )$.
Then the ``distance'' is $T = \| f(\vec v_1) - f(\vec v_0)\| $ and
the two radii are $0$ and $\| g(\vec v_1) - g(\vec v_0)\| $,
so $\beta \le \pi \cdot  \| g(\vec v_1) - g(\vec v_0)\| \div
2 \| f(\vec v_1) - f(\vec v_0)\| \le \theta - \hat\theta$,
using (\ding{"50}).

Regarding $d(f, f+g)$ and referring to Definition \ref{def-dist}:
It is obvious that $\| f - (f + g) \| = \|g\|$, but 
to bound $\| f\iv - (f + g)\iv \|$:  
say $f\iv(y) = x$ and $(f + g)\iv (y) = z$.
Then $f(x) = y = f(z) + g(z)$.  Now
\[
\|g\| \ge \| f(z) - (f(z) + g(z)) \| =
\| f(z) - f(x) \| \ge \|x - z\| / \| J_{f\iv} \| \ \ ,
\]
so $\|x - z\| \le  \|g\| \cdot  \| J_{f\iv} \|$.
\end{proof}

To get such a $g$ that makes $f+g$ linear near a given point, use:

\begin{lemma}
\label{lemma-near-point}
Fix $f \in \FF_\theta$, and assume that $f(\vec 0) = \vec 0$.
Let $A = J_f(\vec 0)$.
Fix any
$\varepsilon > 0$.
Then there is a $C^1$ function $g$
and a $\hat \theta \in (\tw(f), \theta)$ 
such that:
$\| g\| \le \varepsilon$,
and \textup(\ding{"50}\textup) of Lemma \ref{lemma-small-change} holds,
and
$\forall \vec x\; [\|\vec x\| \ge \varepsilon/2 \to g(\vec x) =  \vec 0]$,
and $f(\vec x) + g(\vec x) = A\vec x$ holds in some neighborhood of $\vec 0$,
and $d(f, f+g) \le \varepsilon$.
\end{lemma}

\begin{proof}
Let $M = \max(1, \| J_{f} \|)$ and $L = \max(1, \| J_{f\iv} \|)$.
Then, by Lemma \ref{lemma-lip}, 
$\|f(\vec x_1) - f(\vec x_0) \| \le M \|  \vec x_0 - \vec x_1 \|$
and
$\|f\iv(\vec y_1) - f\iv(\vec y_0) \| \le L \|  \vec y_0 - \vec y_1 \|$
holds for all $ \vec y_0, \vec y_1, \vec x_0, \vec x_1$.
Fix $\hat \theta \in (\tw(f), \theta)$, 
with $\theta - \hat \theta < \pi/2$.  Shrinking $\varepsilon$ if necessary,
we may assume that $\varepsilon \le 2 (\theta -  \hat \theta) /\pi$;
then (\ding{"50}) will follow from:
\[
\| g(\vec v_1) - g(\vec v_0) \| \le (\varepsilon/L) \| \vec v_1 - \vec v_0 \|
\qquad (\forall \vec v_1, \vec v_0 \in \RRR^n) \ \ .  
\tag{\ding{"52}}
\]
Also, $d(f, f+g) \le \|g\| L$ by Lemma \ref{lemma-small-change},
and we shall in fact get $\|g\| \le \varepsilon/L$.

Choose $P,Q,R,\zeta$  with $0 < P < Q < R$ and
$\zeta > 0$, and choose $\psi : \RRR \to [0,1]$ to satisfy:
\begin{itemizz}
\item[a.] $\zeta \le \varepsilon / (2L)$ and $\zeta < 1$.  
\item[b.] $R \le \varepsilon/2$; and
$\| A \vec x - f(\vec x) \| \le \varepsilon/L$ 
and $\| J_f(\vec x) - A \| \le \zeta$ whenever $\|\vec x\| \le R$. 
\item[c.] $Q \le R/2$, and
$\| A \vec x\| + \|f(\vec x) \| \le (\varepsilon/L) (R/2)  $
whenever $\| \vec x \| \le Q$, and
$\zeta + \zeta^2 Q \le \varepsilon / L$.
\item[d.] $\psi$ is $C^\infty$ and non-increasing, and
$\psi(t) = 1$ for all $t \le P$, and
$\psi(t) = 0$ for all $t \ge Q$, and
$ \psi(x) - \psi((1 + \sigma) x)  \le \zeta \sigma$
whenever $\sigma, x > 0$.
\end{itemizz}
There are such $P,\psi$ as in (d) by Lemma \ref{lemma-smooth-log}.
Let $g(\vec x) = \psi(\|\vec x\|) (A \vec x - f(\vec x) )$.
Then $\| g\| \le \varepsilon/L$ by (b).
So, we are done if we verify (\ding{"52}).

Let $r_i = \| \vec v_i \|$.  We may assume that $r_0 \le r_1$.
We may also assume that $r_0 \le Q$, since otherwise (\ding{"52})
is trivial.

If $r_1 \ge R$, then $g(\vec v_1) = \vec 0$ and 
$\| \vec v_0 - \vec v_1 \| \ge (R - Q)$, so it is sufficient to verify
\[
\|  g(\vec v_0) \| \le (\varepsilon/L) (R-Q)  \ \ ,
\]
which follows from (c) above.

From now on, assume that $r_1 \le R$.
Define $\vec w( \vec v_0, \vec v_1 )$ by:
\[
\vec w( \vec v_0, \vec v_1 ) = 
f(\vec v_1) - f(\vec v_0) - A (\vec v_1 - \vec v_0)
= k(\vec v_1) - k(\vec v_0)
\;\; ; \;\;
k(\vec v) = f(\vec v) - A\vec v \ \ .
\]
Note that $J_k = J_f - A$.  Then,
$\| \vec w( \vec v_0, \vec v_1 ) \| \le \zeta \| \vec v_1 - \vec v_0 \|$
whenever $\| \vec v_1\| , \|\vec v_0 \| \le R$; to see this, use
(b) above and Lemma \ref{lemma-lip}.  
Now,
\[
 g(\vec v_1) - g(\vec v_0)  = 
\psi( r_1) (A \vec v_1 - f(\vec v_1) )   - 
\psi( r_0) (A \vec v_0 - f(\vec v_0) )  \ \ .
\]
Let $r = r_0$ and $r_1 = ( (1 + \sigma) r)$.  If $\sigma = 0$,
so $r_0 = r_1 = r$,
then
\[
\begin{array}{ll}
 \| g(\vec v_1) - g(\vec v_0) \| 
&= 
  \psi(r) \| A ( \vec v_1 - v_0 ) - f(\vec v_1) + f(\vec v_0) \| \\
&=
\psi(r) \| \vec w( \vec v_0, \vec v_1 ) \| \le
\zeta \| \vec v_1 - \vec v_0 \| \ \ ,
\end{array}
\]
so (\ding{"52}) holds by (a).
From now on, assume that $\sigma > 0$.  Now
$ g(\vec v_1) - g(\vec v_0)  = $
\[
\psi( r_1)
\big[   
A ( \vec v_1 - \vec v_0 ) - f(\vec v_1) + f(\vec v_0) \big]   +
(\psi(r_1) - \psi( r_0)) (A \vec v_0 - f(\vec v_0) )    
\]
Now $\| A \vec v_0 - f(\vec v_0) \| = \| \vec w ( \vec v_0, \vec 0) \| \le
\zeta Q$ and
$| \psi(r_1) - \psi( r_0) |  = | \psi( (1 + \sigma) r) - \psi( r) |  \le
\zeta \sigma \le \zeta \| \vec v_1 - \vec v_0 \|$, so by the above
argument, 
\[
\| g(\vec v_1) - g(\vec v_0) \|  \le
(\zeta + \zeta^2 Q) \| \vec v_1 - \vec v_0 \| \ \ .
\]
So, we are done by (c) above.
\end{proof}

\begin{proofof}{Lemma \ref{lemma-transl-near-0}}
First, replacing $f$ by $\vec x \mapsto f(\vec x + \vec d) - \vec e$,
it is sufficient to prove the lemma in the case that
$\vec d = \vec e = \vec 0$.  Now, apply
Lemma \ref{lemma-near-point}
and then Lemma \ref{lemma-transl-special}
(both with $\varepsilon/2$ instead of $\varepsilon$).
\end{proofof}

\begin{proofof}{Lemma \ref{lemma-dense-de}}
We show that 
$W := \{p :  \vec d \in \dom(\sigma_p) \}$  is dense.
Fix $p = (\sigma, h, \varkappa, \Ups) \in \PPP^\theta$
with $d \notin \dom(\sigma)$; we shall find
a $q = (\sigma_q, h_q, \varkappa_q, \Ups_q) \le p$ with $q \in W$.
Fix $\ell \in \omega$ such that
$\xi := \hgt(\vec d) + \ell \ne \hgt(\vec z)$ for all
$\vec z \in \dom(\sigma) \cup \ran(\sigma)$.
Let $E_\xi = \{\vec e \in E : \hgt(\vec e) = \xi\}$.

Let $\vec c = h(\vec d)$. 
Fix $\hat \theta \in (\tw(h), \theta)$, 
with $\theta - \hat \theta < \pi/2$.
Let $M = \max(1, \| J_{h} \|)$ and $L = \max(1, \| J_{h\iv} \|)$.
Fix $Q, P, \psi, \varepsilon, \vec e, \vec a$ so that:

\begin{itemizz}
\item[a.]
$Q < \min\{\| \vec d - \vec d' \| : d' \in \dom(\sigma) \}$ and
$\mu(B(\vec 0, MQ)) < \zeta_p$ (see Lemma \ref{lemma-smallchange}).
\item[b.]
$0 < P < Q$ and $\psi: \RRR \to [0,1]$ is a
$C^\infty$ non-decreasing function, and
$\psi(t) = 1$ for all $t \le P$, and
$\psi(t) = 0$ for all $t \ge Q$.
\item[c.]
$0 < \varepsilon <   2 (\theta -  \hat \theta) /(  \pi L \| \psi' \| )$,
and $\varepsilon < \varkappa_p / L$.
\item[d.]
$\vec e = \vec c + \vec a \in E_\xi$ and $\| \vec a\| < \varepsilon$.
\end{itemizz}
Let $h_q ( \vec d + \vec v) = h ( \vec d + \vec v) + \psi(\|\vec v\|) \vec a$;
$h_q \supset \sigma$ by (a).
Let $\sigma_q = \sigma_p \cup \{ (\vec d, \vec e) \}$.
Then $h_q \supset \sigma_q \supset \sigma_p$.
Now apply Lemma \ref{lemma-small-change}, with $f = h$ and
$g( \vec d + \vec v) = \psi(\|\vec v\|) \vec a$.
This yields $h_q \in \FF_\theta$ and
$d(h, h_q) \le L \|g\| \le L \| \vec a \| < L \varepsilon$
(using (d)).

But to see that Lemma \ref{lemma-small-change} applies here,
we need to verify (\ding{"50}); that is,
$\|g(\vec v_1) - g(\vec v_0) \| \le (2 (\theta -  \hat \theta) /\pi)\,
\| h(\vec v_1) - h(\vec v_0) \|$.
Let $r_i = \|v_i\|$; we may assume that $r_0 \le r_1$.
Then $\|g(\vec v_1) - g(\vec v_0) \| \le  \varepsilon \| \psi' \| (r_1 - r_0)$
and
$\| h(\vec v_1) - h(\vec v_0) \| \ge \| \vec v_1 - \vec v_0 | / L \ge
(r_1 - r_0) /L $, so (\ding{"50}) holds by (c).

We obtain $\varkappa_q$ and $\Ups_q$ by using
Lemma \ref{lemma-smallchange}.   This lemma requires both
$d(h,h_q) < \varkappa_p$ (which holds by (c)) and
$\mu(\overline S), \mu(\overline T) \le \zeta_p$.
For this second inequality, apply (a) and note that 
$S \subseteq B(\vec d, Q)$ and
$T \subseteq h(S) \subseteq B(\vec c, MQ)$.

Observe that $\sigma_q \in P^\theta_0$:  (P2) holds
because $\sigma_q \subset h_q$, and
(P3)(P4) hold by (d) and our choice of $\xi$.
\end{proofof}

\begin{proofof}{Lemma \ref{lemma-nice-are-dense}}
If $p \in \PPP^\theta$ and $m = | \sigma_p|$, then we use
Lemma \ref{lemma-transl-near-0} $m$ times to construct
$p = q_0 \ge q_1  \ge q_2 \cdots \ge q_m$, where $q_m$ is nice.
All $q_i$ have the same $\sigma_{q_i} = \sigma_p$,
but $h_{q_i}$ will be a translation in some neighborhood of
$i$ many of the $(\vec d, \vec e) \in \sigma_p$.
Given $q_{i}$, we use Lemma \ref{lemma-transl-near-0} to
construct $h_{q_{i+1}}$ from $h_{q_{i}}$.
But we also make sure that 
$h_{q_{i+1}}$ and $h_{q_{i}}$ are close enough
to be able to use
Lemma \ref{lemma-smallchange} to build an appropriate
$\varkappa_{q_{i+1}}$ and $\Ups_{q_{i+1}}$.
\end{proofof}

The following consequence of Lemma \ref{lemma-small-change} will be useful:

\begin{lemma}
\label{lemma-average}
Fix $\theta \in (0, \pi)$.  To each $f \in \FF_\theta$,
one can assign positive rationals
$\varepsilon_f$ and $\delta_f$ and $M_f$ such that:

Whenever $f,g \in \FF_\theta$ with 
$\delta_f = \delta_g = \delta$ and
$\varepsilon_f = \varepsilon_g = \varepsilon$
and $M_f = M_g = M$:
If $\| f - g \| < \varepsilon$ and
$\| J_f - J_g \| < \delta$ then $(f + g)/2 \in \FF_\theta$.
Furthermore, $d(f, (f + g)/2) \le \|g - f\| \cdot \max(1, \| J_{f\iv} \|)/2$.
\end{lemma}
\begin{proof}
First, let $M_f \ge \max( \|  J_f \| , \| J_{f\iv} | )$.
Then $M_f \ge 1$.  Assume always that $\delta_f < 1/(4M_f)$.

Now use Lemma \ref{lemma-small-change}.
So, $(f + g)/2  = f + h$, where $h = (g - f)/2$.
Choose $\varepsilon_f < 2(\theta - \tw(f))/\pi$.  Then 
(\ding{"50}) from Lemma \ref{lemma-small-change} is satisfied if
$ \| h(\vec v_1) - h(\vec v_0) \| \le
\varepsilon \| f(\vec v_0) - f(\vec v_1) \| $
for all $\vec v_0, \vec v_1 $.

Since $ \| f(\vec v_1) - f(\vec v_0) \| \ge \| \vec v_1 -\vec v_0 \|/ M $,
it is sufficient to ensure that
$ \| h(\vec v_1) - h(\vec v_0) \| \le
(\varepsilon / M) \| \vec v_1 -\vec v_0 \| $.
Let $N = \| J_f - J_g \| = \sup_x \| J_f(x) - J_g(x) \| = 2 \sup_x \| J_h\|$.
Then $ \| h(\vec v_1) - h(\vec v_0) \| \le (N/2) \| \vec v_1 -\vec v_0 \|$,
so just demand that $N < \varepsilon / M$;
that is $\delta_f < \varepsilon_f / M_f$.
\end{proof}

\section{Proof of ccc}
\label{sec-ccc}
Our proof imitates the ccc proof in Lemma \ref{lemma-main-1}.
We start with $p_\alpha$ for $\alpha < \omega_1$ and prove that two
of them are compatible.
By Lemma \ref{lemma-nice-are-dense}, we may assume that
all the $p_\alpha$ are nice.
We now apply some preliminary thinning.
Since there are only 
$\aleph_0$ possibilities for $\varkappa_p$ and $\Ups_p$,
we may assume that each
$p_\alpha = (\sigma_\alpha, h_\alpha, \varkappa, \Ups)$.
We may assume that $| \sigma_\alpha | = t$ for all $\alpha$,
so $\sigma_\alpha = \{
(d_\alpha^0, e_\alpha^0 ), \ldots, (d_\alpha^{t-1} , e_\alpha^{t-1}) \}$.
We may also assume that there is a fixed $\hat \theta \in (\pi/2, \theta)$
such that all $p_\alpha \in \PPP^{\hat \theta}$.

By niceness, WLOG there is a fixed $r > 0$ such that each
$h_\alpha$ is a translation on each $B( d_\alpha^i, r)$;
so $h_\alpha(x) = x + e_\alpha^i - d_\alpha^i $ whenever
$\| x - d_\alpha^i \| \le r$; hence also
$h\iv_\alpha(y) = y + d_\alpha^i - e_\alpha^i $ whenever
$\| y - e_\alpha^i \| \le r$.
We choose our $r$ small enough so that also
$ \|d_\alpha^i - d_\alpha^j\| \gg r$ and
$ \|e_\alpha^i - e_\alpha^j\| \gg r$ whenever $i \ne j$.
WLOG,  the $\sigma_\alpha$ are close to some common
condensation point $\{ (d^0, e^0 ), \ldots, (d^{t-1} , e^{t-1} ) \}$,
so that $\| d_\alpha^i - d^i \| \ll r$ and $\| e_\alpha^i - e^i \| \ll r$,
and hence also $ \|d^i - d^j\| \gg r$ and
$ \|e^i - e^j\| \gg r$ whenever $i \ne j$.

Also, WLOG the $h_\alpha$ are sufficiently close to each other
that Lemma \ref{lemma-average} applies to show that each
$(h_\alpha + h_\beta)/2 \in \FF_{\hat \theta}$.

After a bit more thinning, we apply Lemma \ref{lemma-elem-ccc} to fix
$\alpha\ne \beta$ such that
$\sigma_\alpha$ and $\sigma_\beta$ are compatible in 
$\PP_0^{\hat \theta}$.  Then 
$\sigma := \sigma_\alpha \cup \sigma_\beta \in \PP_0^{\hat \theta}$.
We now construct a $q \in P^\theta$ with $q \le p_\alpha$ and $q \le p_\beta$.
Let $\sigma_q = \sigma$.  Let $\hat h = (h_\alpha + h_\beta)/2$.
Although $\hat h \in \FF_{\hat \theta}$, we cannot let
$\hat h_p = \hat h$ because $\hat h$ need not extend $\sigma$; but it is
``close enough'' to $\sigma$ that we may vary it slightly
to obtain our $h_q \supset \sigma$ with $h_q \in \FF_\theta$.
Finally, we make sure that our $r$ was chosen to be small enough
that the argument of Lemma \ref{lemma-smallchange} can be applied to choose
$\varkappa_q$ and $\Ups_q$.

The hardest part of the argument is modifying $\hat h$ to obtain $h_q$.
We shall have $h_q(x) = \hat h(x)$ unless $x$ is near some
$d_\alpha^i, d_\beta^i$.  More specifically, let
$\hat d^i = (d_\alpha^i + d_\beta^i)/2 $ and
$\hat e^i = (e_\alpha^i + e_\beta^i)/2 $.  Using
$\| d_\alpha^i - d_\beta^i \| \ll r$, we have
\[
\hat h(x) = 
 (x + e_\alpha^i - d_\alpha^i  + x + e_\beta^i - d_\beta^i)/2 =
x + \hat e^i - \hat d^i
\]
when $\|x - d_\beta^i\| \le r/2$, and likewise
$\hat h\iv( y) = y + \hat d^i - \hat e^i$ when $\|y - e_\beta^i\| \le r/2$.
In particular, $\hat h(\hat d^i) =  \hat e^i$.
Then we shall have $h_q(x) = \hat h(x)$ unless 
$\|x - \hat d^i\| \le r/2$ for some $i$,
so the changes are only within the various $B(\hat d^i, r/2)$.
We need to make sure that we can make these changes
without bringing $\tw(h_q)$ above $\theta$.
Using $ \|d^i - d^j\| \gg r$, the changes to $\hat h$ inside the
various $B(\hat d^i, r/2)$ will not interfere with each other.

Focusing on one $i$: if $d_\alpha^i = d_\beta^i$, then 
let $h_q \res B(\hat d^i, r/2) = \hat h \res B(\hat d^i, r/2)$.
Now, assume that $d_\alpha^i \ne d_\beta^i$ and hence
$e_\alpha^i \ne e_\beta^i$;
we need to get $h_q( d_\alpha^i) = e_\alpha^i $
and $h_q( d_\beta^i) = e_\beta^i $.  
Since $\hat h(\hat d^i) = \hat e^i$,
we can temporarily change coordinates in the domain and
range and assume that $\hat d^i = \hat e^i = \vec 0$,
so that now $\hat h(x) = x$ for $x \in B(\vec 0, r/2)$.
Then, let $d = d_\alpha^i$ and $e = e_\alpha^i$,
so $ d_\beta^i = -d$ and $ e_\beta^i = -e$,
and we need to get $h_q(d) = e$ and $h_q(-d) = -e$.
WLOG $K := \|e\| / \| d \| \ge 1$; otherwise, we can
interchange $d,e$ and $\hat h, \hat h\iv$ in the argument.
We remark that there is no a priori upper bound to $K$ in this argument.

In changing $\hat h$ to $h_q$ within $B(\vec 0, r/2)$ we have two tasks:
\emph{expand} and \emph{rotate}:  That is, we must rotate
$d$ by angle $\angle(d,e)$ so that it points in direction $e$;
note that $\angle(d,e) =
\angle( d_\alpha^i - d_\beta^i,  e_\alpha^i - e_\beta^i) \le
\tw(\sigma) < \hat \theta $.
At the same time, we must expand $d$ by 
a factor of $K$ so that it has length $\| e \|$.

The following lemma involves a pure rotation, without expansion:

\begin{lemma}
\label{lemma-step-2}
Given $\pi/2 \le \hat \theta < \theta < \pi$ and $0 < r_0 < r_1$
with $r_1/r_0 > e^{ 5 / (\theta - \hat\theta)}$,
and given $\vec d , \vec e$ with $0 < \| \vec d \| = \| \vec e \| < r_0$
and $\angle (\vec d , \vec e) < \hat \theta$:

There is an $f \in \FF_\theta$ such that $f(\vec d) = \vec e$ and
$f(\vec x) = A( \| \vec x \|) \, \vec x$ for all
$\vec x$, where  $A : \RRR \to \SO(n)$ is a $C^\infty$ function,
with $A(r) = I$ whenever $r \ge r_1$ and
$A(r) = A(0)$ whenever $r \le r_0$.
\end{lemma}
\begin{proof}
Let $\varrho = \angle (\vec d , \vec e) < \hat \theta$.
We may assume that our coordinates are chosen so that
$\vec d, \vec e$ are in the $x_1, x_2$ plane, with $\vec e$
obtained by rotating by $\varrho$ in the positive direction.
Then $A(r) = R_{\psi(r)}$, where $R_\alpha$ is just rotation by $\alpha$
in the $x_1, x_2$ plane, and $\psi \in C^\infty(\RRR, [0, \varrho])$
is a non-increasing function with
$\psi(r) = \varrho$ when $r \le r_0$ and
$\psi(r) = 0$ when $r \ge r_1$.   Then we are done if we show
that we can choose $\psi$ so that $f \in \FF_\theta$.

Let $\zeta = \frac 2 \pi (\theta - \hat\theta) / \varrho$.
Let $\psi = \varrho \psi_0$, where
$\psi_0  \in C^\infty(\RRR, [0, 1])$ is chosen so that
$  \psi_0(r) - \psi_0( (1 + \sigma) r)  \le \zeta \sigma$
whenever $\sigma, r > 0$.  This is possible by
Lemma \ref{lemma-smooth-log} because $r_1 / r_0 > e^{1/\zeta}$.
Now 
$\psi(r) -\psi( (1 + \sigma) r) \le \frac 2 \pi (\theta - \hat\theta) \sigma$
whenever $\sigma, r > 0$.  

To prove that $\tw(f) < \theta$, we fix $\vec x_0 \ne  \vec x_1$,
with $\vec y_i = f(\vec x_i)$, and
show that
$\gamma := \angle ( \vec x_1 - \vec x_0, \vec y_1 - \vec y_0 ) \le
\varrho + (\theta - \hat\theta) $.
If $\| \vec x_0\| = \| \vec x_1\| $
or $\| \vec x_0\| = 0$ then $\gamma \le \varrho$, so
we may assume that
$0 < r = \| \vec x_0\| < (1 + \sigma) r =  \| \vec x_1\| $.
Then $\vec y_0 = R_{\psi(r)} \vec x_0$ and
$\vec y_1 = R_{\psi( (1 + \sigma)r)} \vec x_1$.  Let
$\vec y_0^* = R_{\psi( (1 + \sigma)r)} \vec x_0$.
Then $\angle ( \vec x_1 - \vec x_0, \vec y_1 - \vec y_0^* ) = 
\psi( (1 + \sigma)r) \le \varrho $.
Also, $\|\vec y_1 - \vec y_0^*\| = \sigma r$ and
$\|\vec y_0 - \vec y_0^*\| \le r [ \psi( (1 + \sigma) r) - \psi(r) ]$,
so
$\|\vec y_0 - \vec y_0^*\| / \|\vec y_1 - \vec y_0^*\|\le
\frac 2 \pi (\theta - \hat\theta) $.
Then we use Lemma \ref{lemma-ball} to conclude that
\[
\beta := 
\angle( \vec y_1 - \vec y_0^* , \vec y_1 - \vec y_0) \le
\theta - \hat\theta \ \ ,
\]
and hence that $\gamma \le \varrho + (\theta - \hat\theta) $.
Here, the ``distance'' is $T = \|\vec y_1 - \vec y_0^*\|  \ge \sigma r$ and
the two radii are $0$ and 
$\|\vec y_0 - \vec y_0^*\| \le r [ \psi( (1 + \sigma) r) - \psi(r) ] \le
r[ \frac 2 \pi (\theta - \hat\theta) \sigma ] $, so
$\beta \le \pi \cdot r[ \frac 2 \pi (\theta - \hat\theta) \sigma ] \div
2 \sigma r = \theta - \hat\theta$.
\end{proof}

We next consider the twist of a pure expansion, without rotation:

\begin{lemma}
\label{lemma-step-1}
Assume that $f(\vec x) = \nu( \| \vec x | ) \vec x$,
where $\nu : [0,\infty) \to [0,\infty) $
and the map $r \mapsto \nu(r) r$ is strictly increasing.
Then $\tw(f) \le \pi/2$.
\end{lemma}
\begin{proof}
We prove that
$\gamma := \angle ( \vec x_1 - \vec x_0,
\nu( \| \vec x_1 \| ) \vec x_1 -\nu( \| \vec x_0 \| )  \vec x_0 ) \le \pi/2$
whenever $\vec x_0 \ne  \vec x_1$.
Since $\gamma = 0$ when 
$\| \vec x_0\| = \| \vec x_1\| $ or $\| \vec x_0\| = 0$,
so we may assume that
$0 < r = \| \vec x_0\| <  s =  \| \vec x_1\| $.
Now, we may work entirely in the plane of $\vec x_0, \vec x_1$,
which we identify with $\CCC$, and we may assume that $\vec x_1$ 
is on the positive $x$-axis.  We can now write
$\vec x_0 = r e^{i \delta}$ and $\vec x_1 = s$.  Then
$\nu( \| \vec x_0 \| )  \vec x_0 = r' e^{i \delta}$ and
$\nu( \| \vec x_1 \| )  \vec x_1 = s'$, where $r < r'$ and $s < s'$.
Then
\[
\gamma = \angle( s - r e^{i \delta}, s'- r'e^{i \delta}) =
\angle( 1 - u e^{i \delta} , 1 - v e^{i \delta}) \ \ ,
\]
where $u = r/s < 1$ and $v = r'/s' < 1$.
Then $\gamma < \pi/2$ because both
$ 1 - u e^{i \delta}$ and $1 - v e^{i \delta}$ lie in the same quadrant:
namely quadrant I if $\delta \in (\pi, 2\pi)$ and 
IV if $\delta \in (0, \pi)$.
If $\delta = 0$ or $\delta = \pi$, then $\gamma = 0$.
\end{proof}

Putting these two lemmas together:

\begin{lemma}
\label{lemma-step-1-2}
Given $\pi/2 \le \hat \theta < \theta < \pi$ and
$f \in \FF_{\hat\theta}$ and
$0 < r_0 < r_4$ with
$r_4/r_0 > e^{ 5 / (\theta - \hat\theta)}[2 \pi/ (\theta - \hat\theta)]$
and $f(\vec x) = \vec x$ whenever $\| \vec x\|  \le r_4$,
and given $\vec d , \vec e$ with $0 < \| \vec d \| , \| \vec e \| < r_0$
and $\angle (\vec d , \vec e) < \hat \theta$:

There is a $g \in \FF_\theta$ such that $g(\vec d) = \vec e$ and
$g(\vec x) = f(\vec x)$ whenever $\|\vec x \| \ge r_4$ and
$g(\vec x) = \nu(x) A( \| \vec x \|) \, \vec x$
whenever $\| \vec x \| \le r_4$, where
$A : \RRR \to \SO(n)$ and
$\nu : [0,\infty) \to [0,\infty) $ are $C^\infty$ functions,
and the map $r \mapsto \nu(r) r$ is strictly increasing.
\end{lemma}
\begin{proof}
If $\| \vec d \|  = \| \vec e \| = 0$ then we can let $g = f$,
and the lemma is symmetric in $f, f\iv$, so WLOG
$0 < \| \vec d \| \le  \| \vec e \| < r_0$.  Let $K = 
\| \vec e \| /  \| \vec d \| \in [1, \infty)$.
We remark that it is important for the ccc proof that we 
are given no upper bound to $K$ in this lemma.

Choose $r_1, r_2, r_3$ with 
$r_0 < r_1 < r_2  < r_3 < r_4$  and
$r_1/r_0 > e^{ 5 / (\theta - \hat\theta)}$ and
$r_2/r_1 > 2$ and
$r_4/r_3  > \pi/ (\theta - \hat\theta)$.
Define $s_i = r_i / K$ for $i = 0,1,2$.

Choose $\nu$ so that $\nu(r) = K$ for $r \le s_2$ and
$\nu(r) = 1$ for $r \ge r_3$.  We can make 
$r \mapsto \nu(r) r$ strictly increasing because
$K \cdot  s_2  = r_2 < 1 \cdot r_3$.

As in the proof of Lemma \ref{lemma-step-2},
let $A(r) = R_{\psi(r)}$, where $R_\alpha$ is rotation by angle $\alpha$,
and $\psi \in C^\infty(\RRR, [0, \varrho])$,
where $\varrho = \angle (\vec d , \vec e) < \hat \theta$.
Again, $\psi$ is a non-increasing function; but now 
$\psi(r) = \varrho$ when $r \le s_0$ and
$\psi(r) = 0$ when $r \ge s_1$, and
$\psi(r) -\psi( (1 + \sigma) r) \le \frac 2 \pi (\theta - \hat\theta) \sigma$
whenever $\sigma, r > 0$.  
There is such a $\psi$ by Lemma \ref{lemma-smooth-log} because
$s_1 / s_0 = r_1 / r_0   > e^{\pi \varrho / ( 2  (\theta - \hat\theta) )  }$.

This defines $g$.  To prove that $g \in \FF_\theta$, 
we fix $\vec x_0 \ne  \vec x_1$,
with $\vec y_i = g(\vec x_i)$, and
show that
$\gamma := \angle ( \vec x_1 - \vec x_0, \vec y_1 - \vec y_0 ) < \theta$.
We may assume that $\|\vec x_0\| \le \|\vec x_1\|$, and we consider various
cases for the values of $\|\vec x_0\| , \|\vec x_1\|$:

If $\|\vec x_1\| \ge r_4$ and $\|\vec x_0\| \ge r_3$, then
$\gamma < \hat \theta$ because $\tw(f) < \hat \theta$ and
each $\vec y_i = g(\vec x_i) = f(\vec x_i)$.

If $\|\vec x_1\| \ge r_4$ and $\|\vec x_0\| \le r_3$:
Then $\|\vec y_1\| \ge r_4$ and $\|\vec y_0\| \le r_3$ 
and also
$\angle ( \vec x_1 - \vec 0, \vec y_1 - \vec 0 ) =
\angle ( \vec x_1 - \vec 0, f(\vec x_1)  - f(\vec 0) ) < \hat \theta$ because
$\tw(f) < \hat \theta$.
We shall show that
\[
\angle ( \vec x_1 - \vec 0, \vec x_1 - \vec x_0 ) < ( \theta - \hat \theta)/2
\quad \text{and} \quad
\angle ( \vec y_1 - \vec 0, \vec y_1 - \vec y_0 ) < ( \theta - \hat \theta)/2
\ \ .
\]
Applying Lemma \ref{lemma-ball},
the ``distance'' $T$ is either
$\|  \vec x_1 - \vec 0 \|$ or $\|  \vec y_1 - \vec 0 \|$,
so $T \ge r_4$, and
the two radii are $0$ and one of $\|\vec x_0\|$,$\|\vec y_0 \|$,
so each of the two angles is bounded by
$\pi \cdot r_3 \div 2 r_4 <  ( \theta - \hat \theta)/2$ because
$r_3/r_4  <  (\theta - \hat\theta) / \pi$.

In the remaining cases, $\|\vec x_0\| \le \|\vec x_1\| \le r_4$.

If $s_1 \le \|\vec x_0\| \le \|\vec x_1\| \le r_4$, then
$\gamma \le \pi/2 < \theta$ by Lemma \ref{lemma-step-1}.

If $0 \le \|\vec x_0\| \le \|\vec x_1\| \le s_2$, then
$\gamma  < \theta$ by Lemma \ref{lemma-step-2}.

All that remains is the case that $0 \le \|\vec x_0\| \le s_1$
and $s_2  \le \|\vec x_1\| \le r_4$:
Then $\angle ( \vec x_1 - \vec 0, \vec y_1 - \vec 0 ) = 0$.
Also, $\angle ( \vec x_1 - \vec 0 , \vec x_1 - \vec x_0)  \le
(\pi/2) (s_1/s_2)$ and
$\angle ( \vec y_1 - \vec 0 , \vec y_1 - \vec 0) \le
(\pi/2) (r_1/r_2)$, so $\gamma < \pi/2$
because $s_2/s_1 = r_2/r_1 > 2$.

Note that our argument requires no lower bound to $r_3/r_2$;
we just need $r_2 < r_3$.  If $r_2 \approx r_3$ then 
$\| J_{f\iv}(y) \| \gg 1$ for some $y$ with $r_2 < y < r_3$,
but our proof does not maintain any upper bound on
$\| J_f\|$ and $\| J_{f\iv} \|$ anyway.
\end{proof}

Before we choose $\varkappa_q$ and $\Ups_q$, we need
some more preliminaries:

\begin{definition}
\label{def-dist-Delta}
For $f,g \in \FF_\theta$, let \\
$\Delta(f,g) = 
\max( \| f - g \|  , \| J_f - J_g \| , 
\| f\iv - g\iv \|  , \| J_{f\iv} - J_{g\iv} \| )$.
Let 
$B_\Delta(f,\varepsilon) = \{g \in \FF_\theta :
d_\Delta( f , g ) < \varepsilon\}$, and
$ B(f,\varepsilon) =  B_d(f,\varepsilon) =
\{g \in \FF_\theta : d( f , g ) < \varepsilon\}$.
\end{definition}

Note that $\Delta(f,g) \ge d(f,g)$ (see Definition \ref{def-dist}),
and $\Delta(f,g) = \Delta(f\iv, g\iv)$,
and $d(f,g) = d(f\iv, g\iv)$
Also, both  $(\FF_\theta, d)$ and $(\FF_\theta, \Delta)$
are separable metric spaces, and neither is complete.
Although $\Delta$ might seem more ``natural'' then $d$ as
a metric on our space $\FF_\theta$ of $C^1$ functions,
our generic function $f$ cannot be $C^1$, and is a limit of
$\langle h_p : p \in G \rangle$ only with respect to $d$, not $\Delta$.

\begin{lemma}
\label{lemma-continuity}
Give $\FF_\theta$ the topology induced by $\Delta$,
and $\MM^n$ the topology induced by the operator norm.
Then the map $f \mapsto f\iv$ is a homeomorphism (and isometry) of $\FF_\theta$,
and the maps $f,x \mapsto J_f(x)$ and $f,x \mapsto J_{f\iv}(x)$ are
continuous $\FF_\theta \times \RRR^n \to \MM^n$.

If $0 < a < b < \infty$, let $\MM^n_{a,b} =
\{Y \in \MM^n : a \le \det Y \le b \}$.
This  $\MM^n_{a,b}$ is closed in $\MM^n$ but not compact.
But, the map
$Y \mapsto Y\iv$ is a uniformly continuous bijection from
$\MM^n_{a,b}$ onto $\MM^n_{1/b,1/a}$.
\end{lemma}
\begin{proof}
For the last statement, use by the standard formula for $Y\iv$ as
a polynomial in the entries of $Y$ divided by $\det(Y)$.
\end{proof}

\begin{definition}
Fix $f \in \FF_\theta$, and let $K \subset (1, \infty)$
be closed in $\RRR$.  Then
$Z^f_K = \{x \in \RRR^n :   \det J_f(x)   \in K \}$.
\end{definition}

Note that $Z^f_K $ is compact because $J_f(x) = I$ outside a bounded set.
Also,
$Z^f_\ell = Z^f_{[\ell - 1, \ell]  }$ ($\ell \ge 3$) and
$W^f_\ell = Z^f_{[\ell,  \infty)  }$ ($\ell \ge 2$)
(see Definition \ref{def-WZ}).

\begin{lemma} 
\label{lemma-basic-Z}
Fix  $K \subset (1, \infty)$ such that $K$ is closed in $\RRR$:

1. For all $\zeta > 0$, there is an open $U \supset K$ such that
$\mu( Z^f_{\overline U} ) \le \mu( Z^f_K ) + \zeta $.

2. $\forall f \in \FF_\theta \ \forall \zeta > 0 \,
\exists \varepsilon > 0 \, \forall g \in \FF_\theta
\; [ \Delta(f,g) < \varepsilon \to  \\
\mu( Z^g_K ) < \mu( Z^f_K) + \zeta \;\wedge\;
\mu( Z^{g\iv}_K ) < \mu( Z^{f\iv}_K) + \zeta]$.
\end{lemma}
\begin{proof}
For (1):  Get open $U_m \supset K$ with all $\overline U_m \subset (1, \infty)$
and $\overline U_m \searrow K$.   Then
$\mu( Z^f_{\overline U_m} ) \searrow \mu( Z^f_K )$ because
the $\mu( Z^f_{\overline U_m} )$ are finite.

For (2):  By symmetry between $f,f\iv$, we need only consider
the ``$\mu( Z^g_K ) < \mu( Z^f_K) + \zeta$'' part of the conjunction.
Note that the ``$<$'' might be \emph{much} less;
for example, $K$ may be a singleton with $\mu( Z^f_K) > 0$;
but there will be $g$ arbitrarily close to $f$ with
$ Z^g_K = \emptyset$.

First fix an open $U$ with
$K \subseteq U  \subseteq \overline U \subseteq (1,\infty)$ and
$\mu( Z^f_{\overline U} ) < \mu( Z^f_K ) + \zeta $.
Then it is sufficient to choose $\varepsilon> 0$ so that
$ \forall g \in \FF_\theta \; [ \Delta(f,g) < \varepsilon \to
Z^g_K  \subseteq  Z^f_{\overline U} ] $.
First fix $r > 0$ such that $J_f(x) = I$ whenever $\| x \| \ge r$.
Then, since $K \subset (1,\infty)$
we can fix $\varepsilon_0 > 0$ such that
$J_g(x) \notin K$ whenever $\| x \| \ge r$ and $\Delta(f,g) < \varepsilon_0$.
Then we shall choose our desired $\varepsilon$
so that $\varepsilon \le \varepsilon_0$.
If there is no such $\varepsilon$, then get a sequence
$g_m \to f$ (wrt $\Delta$) and
$x_m \in  Z^{g_m}_K  \setminus Z^f_{\overline U}$
with all $\Delta(f, g_m) < \varepsilon_0$.
Then all $\| x_m \| < r$, so,
passing to a sub-sequence, we may assume that $x_m \to x$.
Then $\det J_{f}(x_m)   \to \det J_{f}(x)  \notin U$ (since $U$ is open)
and $\det J_{g_m}(x_m)  \to  \det J_{f}(x)  \in K$
(since $K$ is closed), which contradicts $K \subseteq U$.
\end{proof}

Note that $K$ need not be bounded here; in particular,
it could be some $[c, \infty)$ with $c > 1$, so this lemma
applies to the $W^f_\ell$.

$\FF_\theta$ is not closed under $+$, since
for $f,g \in \FF_\theta$, $f + g$ need not even be 1-1.
Of course, $f \in \FF_\theta$ implies $c f \in \FF_\theta$ for any $c > 0$.
Also note the following related to Lemmas
\ref{lemma-small-change} and \ref{lemma-average}.
In both of them, we are starting with an $f \in \FF_\theta$
and we are constructing a new function $k \in \FF_\theta$,
and we easily verify that 
$\| f - k \|$, $\| f\iv - k\iv \|$, and $\| J_f - J_k \|$ are ``small'',
and we want to show that  $\| J_{f\iv} - J_{k\iv} \|$ is ``small'', so that
$\Delta(f,k)$ is small.
Applied to Lemma \ref{lemma-average}, $k = (f + g)/2 $, where
$g$ is ``near to'' $f$.

\begin{lemma}
\label{lemma-bound-inv}
For each $f \in \FF_\theta$,  and each $\varepsilon > 0$,
there is a $\delta \in (0, \varepsilon)$ such that:
For all $k \in \FF_\theta$, 
if $\| f - k \| < \delta$, $\| f\iv - k\iv \| < \delta$,
and $\| J_f - J_k \| < \delta$, then
$\| J_{f\iv} - J_{k\iv} \| < \varepsilon$, and hence
$\Delta(f,k) < \varepsilon$.
\end{lemma}
\begin{proof}
To bound $\| J_{f\iv} - J_{k\iv} \|$, fix $y$ and we bound
$\| J_{f\iv}(y) - J_{k\iv}(y) \|$.
Let $f\iv(y) = x$ and $k\iv(y) = z$, so $f(x) = k(z) = y$, and
$\| J_{f\iv}(y) - J_{k\iv}(y) \| = \| (J_f(x))\iv - (J_k(z))\iv \| \le
 \| (J_f(x))\iv - (J_f(z))\iv \| + \| (J_f(z))\iv - (J_k(z))\iv \|$.

For the first summand:  Given $f$, the maps
$x \mapsto J_f(x)$ and $x \mapsto (J_f(x))\iv$ are continuous on 
$\RRR^n$, and hence uniformly continuous
(since $J_f(x) = I$ outside a bounded set),
so choose $\delta > 0$ small enough that
$\|x - z\| < \delta \to \| (J_f(x))\iv - (J_f(z))\iv \| < \varepsilon / 2$.
Note that
$\|x - z\| = \|f\iv(y) - k\iv(y) \| \le \| f\iv - k\iv \| < \delta$.

For the second summand, let $2a$ be the \emph{smallest} value
of $\det(J_f(z))$.  Choosing $\delta$ small enough yields
$\| J_f - J_k \| < \delta \to  \det(J_k(z)) \ge a$.
Using this plus Lemma \ref{lemma-continuity},
we can choose $\delta$ so that $\| J_f - J_k \| < \delta \to 
\| (J_f(z))\iv - (J_k(z))\iv \| > \varepsilon / 2$.
\end{proof}

We remark that the proof for the first summand
is not uniform on $f$, and our $\delta$ really depends on $f$,
because our functions are only $C^1$, not $C^2$, so 
the maps $x \mapsto J_f(x)$ and $x \mapsto (J_f(x))\iv$ are continuous 
but necessarily Lipschitz.
But if we worked with $C^2$ functions, then our $\Delta$ would
need to use the second derivatives, so we would have the same problem
one level up.

\begin{proofof}{Lemma \ref{lemma-ccc}}
We begin with the details of the thinning argument.
We start with $p_\alpha =
(\sigma_\alpha, h_\alpha, \varkappa_\alpha, \Ups_\alpha)$,
for $\alpha < \omega_1$.  with $m_\alpha = \dom(\sigma_\alpha)$.  Then,

\begin{enum}
\item\label{thinA}
WLOG, all $m_\alpha \ge 4$, and all $\| \sigma_\alpha\| \ge 1$,
and all $p_\alpha$ are nice.
\item\label{thinB}
WLOG:  
all $\Ups_\alpha$  are the same $\Ups$;
and
all $\varkappa_\alpha$ are the same $\varkappa$;
so $p_\alpha = (\sigma_\alpha, h_\alpha, \varkappa, \Ups)$;
and all $|\sigma_\alpha| = t \ge 1$.
Let $m = \dom(\Ups) \ge 4$, and let
$\sigma_\alpha = \{(d_\alpha^i, e_\alpha^i) : i < t\}$.
\item\label{thinD}
$\hat \theta \in (\pi/2, \theta)$, and WLOG all $p_\alpha \in \PPP^{\hat \theta}$
and all $(h_\alpha + h_\beta)/2 \in \FF_{\hat \theta}$.
\item\label{thinE}
WLOG: there is a fixed $r > 0$ such that each
$h_\alpha$ is a translation on each $B( d_\alpha^i, r)$;
so $h_\alpha(x) = x + e_\alpha^i - d_\alpha^i $ whenever
$\| x - d_\alpha^i \| \le r$; hence also
$h\iv_\alpha(y) = y + d_\alpha^i - e_\alpha^i $ whenever
$\| y - e_\alpha^i \| \le r$.
\item\label{thinG}
WLOG:
there is some fixed rational $\varepsilon > 0$ such that
$\mu(Z^{h_\alpha}_\ell) < \Ups(\ell) - \varepsilon$ and
$\mu(Z^{h_\alpha\iv}_\ell) < \Ups(\ell) - \varepsilon$
holds for each $\alpha$ whenever $3 \le \ell  < m$,
and $\sum \{ \ell \Ups(\ell) :
3 \le \ell <  m \} < 1 -  \varepsilon$, and
$Z^{h_\alpha}_{[m - 1 -\varepsilon, \infty)} = 
Z^{h_\alpha\iv}_{[m - 1 -\varepsilon, \infty)} = \emptyset$.
\item\label{thinH}
$\sigma = \{ (d^0, e^0 ), \ldots, (d^{t-1} , e^{t-1} ) \}$
is a condensation point of $\{\sigma_\alpha : \alpha < \omega_1\}$
(considering these $\sigma_\alpha$ as points in $(\RRR^n)^{2t}$),
and $h$ is a condensation point of $\{h_\alpha : \alpha < \omega_1\}$
(with respect to the metric $\Delta$).
Also, $\sigma \in \PPP_0^{\hat \theta}$ and
$h \in \FF_{\hat \theta}$ and
$\mu(Z^{h}_\ell) < \Ups(\ell) - \varepsilon$ and
$\mu(Z^{h\iv}_\ell) < \Ups(\ell) - \varepsilon$
whenever $3 \le \ell  < m$, and
$Z^{h}_{[m - 1 -\varepsilon, \infty)} = 
Z^{h\iv}_{[m - 1 -\varepsilon, \infty)} = \emptyset$.
\item\label{thinF}
 WLOG:
$ \|d^i - d^j\| > 8\pi r / (\theta-\hat\theta)$ and
$ \|e^i - e^j\| > 8\pi r / (\theta-\hat\theta)$ whenever $i \ne j$,
and $\mu(B(\vec 0, r)) < \varepsilon/(2t)$.
Also, $r < \varkappa/8$.
\item\label{thinI}
$\nu$ is small enough so that
for all $g \in \FF_\theta$, if
$ \Delta(g,h) < \nu$ then
$ \mu( Z^g_\ell ) < \mu( Z^h_\ell) + \varepsilon/2$ and
$\mu( Z^{g\iv}_\ell ) < \mu( Z^{h\iv}_\ell) + \varepsilon/2$
whenever $3 \le \ell  < m$.
Also, for all such $g$,
$ Z^{g}_{[m - 1 -\varepsilon/2, \infty)}  = \emptyset$ and
$ Z^{g\iv}_{[m - 1 -\varepsilon/2, \infty)}  = \emptyset$.
Also, $\nu < \varkappa/8$.
\item\label{thinJ}
$\WW$ is an open neighborhood of $h$ in $\FF_{\hat \theta}$,
and $\forall f,g \in \WW\, [\Delta( h, (f + g) / 2 ) < \nu ] $, 
and WLOG all $h_\alpha \in \WW$.
\item\label{thinK}
Let $r_4 = r/2$, and choose $r_0 \in (0, r_4)$ so that
$r_4/r_0 > e^{ 5 / (\theta - \hat\theta)}[2 \pi/ (\theta - \hat\theta)]$.
WLOG,
$\| d_\alpha^i - d^i \| < r_0/8$ and $\| e_\alpha^i - e^i \| < r_0/8$
for all $\alpha, i$.
\end{enum}

To justify some of these steps:

For \pref{thinA}:  use the facts that $\{p : m_p \ge 4\}$ is dense
(Lemma \ref {lemma-dense-m}), and 
$\{p : |\sigma_p|  \ge 1\}$ is dense
(e.g., by Lemma \ref{lemma-dense-de}),
and the nice $p$ are dense (Lemma \ref{lemma-nice-are-dense}).

For \pref{thinG}: 
Note that $\sup_x \det J_{h_\alpha}(x) = \max_x \det J_{h_\alpha}(x) < m-1$,
using Definition \ref{def-poset} and the fact that
$\det J_{h_\alpha}(x) = 1$ outside a bounded set.

For \pref{thinH}: use separability of the spaces involved.
To ensure that $\sigma \in \PPP_0^{\hat \theta}$ and
$h \in \FF_{\hat \theta}$, etc., we may
take $\sigma$ to be one of the $\sigma_\alpha$
and take $h$ to be one of the $h_\alpha$.

For \pref{thinF}: shrink $r$ if necessary.

For \pref{thinI}, see Lemma \ref{lemma-basic-Z}, and
for \pref{thinJ}, see Lemma \ref{lemma-bound-inv}.
Regarding getting $ Z^{g}_{[m - 1 -\varepsilon/2, \infty)}  = \emptyset$:
we have $\forall x \, [\det J_h(x) < m -1 - \varepsilon]$, so if
$\|J_g -  J_h\|$ is small enough, we'll have
$\forall x \, [\det J_g(x) < m -1 - \varepsilon/2]$.

We remark that the $r_0, r_4$ in \pref{thinK} corresponds
to the $r_0, r_4$ in Lemma \ref{lemma-step-1-2}.

Now, to verify the ccc, 
fix $\alpha\ne \beta$ such that
$\sigma_\alpha$ and $\sigma_\beta$ are compatible in 
$\PPP_0^{\hat \theta}$.
Then $\sigma := \sigma_\alpha \cup \sigma_\beta \in \PPP_0^{\hat \theta}$.
We show that $p \compat q$ (in $\PPP^{\theta}$) by constructing
a $q \in \PPP^{\theta}$ such that $q \le p_\alpha$ and $q \le p_\beta$.
Let $\sigma_q = \sigma$.  Let $\hat h = (h_\alpha + h_\beta)/2$.
We must modify $\hat h$ to obtain $h_q$.
To do this, we apply Lemma \ref{lemma-step-1-2} $t$ times.

Let
$\hat d^i = (d_\alpha^i + d_\beta^i)/2 $ and
$\hat e^i = (e_\alpha^i + e_\beta^i)/2 $.  Then
$\hat h(\hat d^i) = \hat e^i$, and $\hat h$ is translation,
$\hat h(x) = x + \hat e^i - \hat d^i$, mapping
$B(\hat d^i, r_4)$ onto $B(\hat e^i, r_4)$.
Also, by
\pref{thinI}\pref{thinJ},
$ \Delta( h, \hat h ) < \nu $ and
$ \mu( Z^{\hat h}_\ell ) < \mu( Z^h_\ell) + \varepsilon/2$ and
$\mu( Z^{{\hat h}\iv}_\ell ) < \mu( Z^{h\iv}_\ell) + \varepsilon/2$
whenever $3 \le \ell  < m$.
Hence, 
$ \mu( Z^{\hat h}_\ell ) <  \Ups(\ell) - \varepsilon/2$ and
$\mu( Z^{{\hat h}\iv}_\ell ) <  \Ups(\ell) - \varepsilon/2$

We also have 
$\mu( Z^{\hat h}_{[m-\varepsilon, \infty)} ) < \varepsilon/2$ and
$\mu( Z^{\hat h\iv}_{[m-\varepsilon, \infty)} ) < \varepsilon/2$.

We let $h_q(x) = \hat h(x)$ for $x \notin \bigcup_i B(\hat d^i, r_4)$.
For each $i$, 
$h_q\res B(\hat d^i, r_4)$ is obtained from
$\hat h\res B(\hat d^i, r_4)$ by one application of Lemma \ref{lemma-step-1-2}
(temporarily
changing coordinates and assuming that $\hat d^i = \hat e^i = \vec 0$).
Now that we have $h_q$, we must verify that 
$\angle(  \vec x_1  - \vec x_0 ,  h_q(\vec x_1)  - h_q(\vec x_0)) < \theta$.
This can only be a problem if 
$\vec x_0  \in B(\hat d^i, r_4)$ and
$\vec x_1  \in B(\hat d^j, r_4)$ for $i \ne j$.
$\| d_\alpha^i - d^i \| < r_0/8$ and
$\| d_\beta^i - d^i \| < r_0/8$, so
$\| \hat d^i - d^i \| < r_0/8$.
Thus, 
$\| x_0 - d^i \| < r$, and likewise $\| x_1 - d^j \| < r$, while
$ \|d^i - d^j\| > 8\pi r / (\theta-\hat\theta)$.
Applying Lemma \ref{lemma-ball}, we conclude that
$\beta := \angle( d^j - d^i, x_1 - x_0) \le (\theta - \hat \theta) / 8$
(so we are done by using $\tw(\sigma)  < \hat \theta$).
Here, the ``distance'' $T = \|d^i - d^j\| > 8\pi r / (\theta-\hat\theta)$,
and the two radii are $< r$, so  Lemma \ref{lemma-ball} says that
$\beta \le \pi \cdot 2r \div  16\pi r / (\theta-\hat\theta)   $.

Finally, we choose $\varkappa_q$ and $\Ups_q$
using the method of proof of Lemma \ref{lemma-smallchange};
see also the corresponding argument in the proof of Lemma \ref{lemma-main-1}.

For $q \le p$, we need
$\varkappa_q \le \varkappa$
and
$B(h_q, \varkappa_q) \subseteq
B(h_\alpha, \varkappa) \cap B(h_\beta, \varkappa)$,
and these are satisfied if we choose
$\varkappa_q < \varkappa - \max( d(h_q, h_\alpha), d(h_q, h_\beta) )$;
this number is positive by \pref{thinF}\pref{thinI}\pref{thinJ}.

Also, for $q$ to be in $\PPP^\theta$, we are required
to choose $m_q \ge  m$ so that
$1/ (m_q -1) <   \det J_{h_q}(x) < ( m_q -1)$ for all $x$;
then, for $m \le \ell < m_q$, we need to choose $\Ups_q(\ell)$ to satisfy:
$\sum \{ \ell   \Ups_q(\ell) : \ell \ge 3 \ \&\ \ell < m_q \} < 1$,
as well as 
$\mu(Z^{h_{q}}_\ell) < \Ups_q(\ell)$ and
$\mu(Z^{h_{q}\iv}_\ell) < \Ups_q(\ell)$ 
whenever $3 \le \ell  < m_q$.
When $\ell < m$, this is guaranteed by \pref{thinI}.
When $m \le \ell < m_q$, we use \pref{thinI}\pref{thinJ} to bound
$\mu(Z^{\hat h}_\ell)$ and $\mu(Z^{\hat h\iv}_\ell)$,
and then we use \pref{thinF} plus the fact that
$\hat h$ and $h_q$ agree outside a set of measure
no more than $t \cdot \mu(B(\vec 0, r)) < \varepsilon/2$.
\end{proofof}

Observe that in building $h_q$ from $\hat h$, we lose any bound
that we had on the Jacobians; in particular, 
$d(h_q,  \hat h)$ is small but $\Delta(h_q,  \hat h)$ isn't.

\section{Examples and Remarks}
\label{sec-examples}
We provide here the examples mentioned in the previous sections.

The following shows that
the ``$\theta > \pi/2$'' in Proposition \ref{prop-intro}
cannot be replaced by ``$\theta \ge \pi/2$'':

\begin{example}
\label{ex-twistlimit}
There are $\aleph_1$--dense $D,E \subseteq \RRR^2$
such that no bijection $f : D \to E$ satisfies $\tw(f) \le \pi/2$.
\end{example}
\begin{proof}
Let $E = \widehat E \times \widehat E$, where
$\widehat E$ is an $\aleph_1$--dense subset of $\RRR$.
Let $D \subseteq \RRR^2$  be any 
$\aleph_1$--dense set of the form
$\bigcup_{n \in \omega} \widehat D_n \times \{y_n\}$,
where each $\widehat D_n \subseteq \RRR$.

Now, fix a 1-1 function  $f : D \to E$ with $\tw(f) \le \pi/2$, and we shall
show that $f$ is not onto.  For this, it is sufficient to show that for each
$n \in \omega$, there is a countable $A_n \subseteq \widehat E$
such that $| ( f( \widehat D_n \times \{y_n\})\, )_t |  \le 1$
for all $t \in \widehat E \backslash A_n$;
here, $(X)_t = \{ u : (t,u) \in X \}$.

Fix $n$.  For $x \in \widehat D_n$, let
$f(x,y_n) = (g_n(x), h_n(x))$, where $g_n, h_n : \widehat D_n \to \widehat E$.
Then $g_n : \widehat D_n \to \RRR$ is non-decreasing
(using $\tw(f) \le \pi/2$), so each $g_n\iv \{t\}$ is a convex
subset of $\widehat D_n$, so 
$A_n := \{t : | g_n\iv \{t\} | \ge 2 \}$ is countable.
If $t \in \widehat E \backslash A_n$, then there is at most
one $x$ such that $g_n(x) = t$,
which implies that
$| ( f( \widehat D_n \times \{y_n\})\, )_t |  \le 1$.
\end{proof}

\begin{example}
\label{ex-no-op}
In Example \ref{ex-twistlimit}, $D$ and $E$ can be taken so that
the two coordinate projections 
$\pi_0$ and $\pi_1$ are both 1-1 on $D$ and on $E$.
Note that no bijection $f : D \to E$ is order-preserving on each
coordinate \textup(i.e., $\pi_i(d') < \pi_i(d) $ iff
$\pi_i(f(d')) < \pi_i(f(d)) $ for $i = 0,1$\textup).
\end{example}
\begin{proof}
To get $D,E$, start with $D_0, E_0$ satisfying
Example \ref{ex-twistlimit}, and obtain $D,E$ by rotating $D,E$ by
some angle $\alpha$ chosen to make $\pi_0, \pi_1$ 1-1.  Such an 
$\alpha$ obviously exists under $\neg \CH$, but in any case, it is
easy to choose $\widehat E$ and the  $\widehat D_n$ and the
$y_n$ in the proof so that $\alpha = 40^\circ$ works.

For the ``note that'', observe that if 
$\angle(d' - d, e' - e) \ge \pi/2$, then
$d' - d$ and $e' - e$ lie in different quadrants.
\end{proof}

We next point out that 
Proposition \ref{prop-intro}, and hence also Theorem \ref{thm-main},
cannot be proved from $\MA + \cccc = \aleph_2$ alone:

\begin{example}
\label{ex-entangled}
It is consistent with $\MA + \cccc = \aleph_2$ that there
are $\aleph_1$--dense
$D,E \subset \RRR^2$ such that $\pi \in \twist(f)$ whenever
$f$ is a bijection from $D$ onto $E$.
\end{example}
\begin{proof}
Work in a model of $\MA + \cccc = \aleph_2$ in which there is
a 2-entangled subset of $\RRR$ of size $\aleph_1$
(see \cite{ARS,AS}), and partition this set into disjoint
pieces $A_q$ and $B_q$ for $q \in \QQQ$.  We may assume that all
$A_q$ and $B_q$ are $\aleph_1$--dense in $\RRR$.

Then, let $D = \bigcup_q A_q \times \{q\}$ and
$E = \bigcup_q B_q \times \{q\}$.  Say $f: D \to E$ is a bijection.
Then fix $q,r \in \QQQ$ and
$\widehat A \in [A_q]^{\aleph_1}$ and
$\widehat B \in [B_q]^{\aleph_1}$ and a bijection
$g : A \to B$ such that the map $(x,q) \mapsto (g(x), r)$
is a sub-function of $f$.
By entangledness, $g$ is not order-preserving, so choose $a < a'$ in $A$
such that $g(a) > g(a')$. 

If $d = (a,q)$ and $d' = (a',q)$ then
$\angle(d' - d, f(d') - f(d)) = \pi$.
\end{proof}

It is easy to modify Examples
\ref{ex-twistlimit}, \ref{ex-no-op}, and \ref{ex-entangled} to
replace $\RRR^2$ by $\RRR^n$ for any $n \ge 2$.

\begin{question}
Forcing with $\PPP^\theta_0$, with $\theta \in (\pi/2, \pi)$,
are 
$\{p : d \in \dom(p)\}$ and $\{p : e \in \ran(p)\}$ dense for
all $d \in D$ and $e \in E$?
\end{question}

If the answer is ``yes'', then we could dispense with the side conditions
in the proof of Proposition \ref{prop-intro}, resulting in a much simpler
proof, but we needed the side conditions anyway in the proof
of Theorem \ref{thm-main} to ensure that the
generic function is BAC.

The interest of this question for forcing is only when $\theta > 90^\circ$,
but a simple example in the plane shows that the answer is ``no''
with $\theta = 18^\circ$:
Let $p = \{(d_i, e_i) : i < 3\}$, where
$d_0 = (0,10)$,
$e_0 = (0,-9)$,
$d_1 = e_1 = (0,-10)$, and
$d_2 = e_2 = (0,11)$.
Then $\tw(p) = 0$, so $p \in \PPP^\theta$.
Let $d = (10,0)$ and suppose that
$p \cup \{(d,e)\} \in \PPP_0^{\theta}$.
Let $e = (x,y)$.
The requirements $\angle( d - d_0, e - e_0) \le 18^\circ$ and
$\angle( d - d_1, e - e_1) \le 18^\circ$ imply that
$0 \le x \le 1$ and $-10 \le y \le -9$.
But then we have
$\angle( d - d_2, e - e_2) \ge
\angle( (10,0) - (0,11),\; (1,-9) - (0,11)) \approx 39^\circ$.

\end{document}